\documentclass[leqno]{amsproc}
\usepackage{amsmath,amsthm,amssymb,amsfonts}
\usepackage{graphicx}
\usepackage{color}
\usepackage[utf8]{inputenc}

\newtheorem{theorem}{Theorem}[section]
\newtheorem{lemma}[theorem]{Lemma}

\newtheorem{proposition}[theorem]{Proposition}
\newtheorem{notation}[theorem]{Notation}

\newtheorem{corollary}[theorem]{Corollary}
\theoremstyle{definition}
\newtheorem{definition}[theorem]{Definition}
\newtheorem{remark}[theorem]{Remark}
\numberwithin{equation}{section}
\def\N{{\mathbb N}}

\def\R{{\mathbb R}}

\def\C{{\mathbb C}}
\def\co{\hbox{\rm co}}

\def\ext{\hbox{\rm Ext}}
\def\llll{{\longrightarrow}}
\def\sep{{ \ \ }}
\def\sem{{\ \ \ \ }}
\def\seg{{\ \ \ \ \ \ }}

\def\com#1{{``#1''}}
\title[The Bishop-Phelps-Bollob{\'a}s property] {Characterization of   Banach spaces $Y$ satisfying that \\ the pair $ (\ell_\infty^4,Y )$ has the Bishop-Phelps-Bollob{\'a}s \\ property  for operators
}

\author[M.D. Acosta]{Mar\'{\i}a D. Acosta}
\address{Universidad de Granada, Facultad de Ciencias,
	Departamento de An\'{a}lisis Matem\'{a}tico, 18071 Granada, Spain}
\email{dacosta@ugr.es}
\author[J.L. D{\'a}vila]{Jos{\'e} L. D{\'a}vila}
\address{Departamento de Matem\'{a}tica,  Facultad de Ciencias, Universidad de Los Andes, M{\'e}rida, 5101 Venezuela}
\email{jldavila@ula.ve}
\author[M. Soleimani]{Maryam Soleimani-Mourchehkhorti}
\address{School of Mathematics, Institute for Research in Fundamental Sciences (IPM), P.O. Box: 19395-5746, Tehran, Iran}
\email{m.soleymanei@sci.ui.ac.ir}

\thanks{The  first  author was  supported  by Junta de Andaluc\'{\i}a grant  FQM--185  and also by Spanish MINECO/FEDER grant  MTM2015-65020-P. The second author was also partially supported by Junta de Andaluc\'{\i}a grant  FQM--185. The third author was supported   by a grant from IPM}

\begin{document}

   {\large

%{\bf Abstract}

\begin{abstract}
We study the Bishop-Phelps-Bollob\'as  property for operators from $\ell_\infty ^4 $ to a Banach space.   For this reason we introduce an appropiate geometric property, namely the AHSp-$\ell_\infty ^4$.   We prove  that spaces $Y$satisfying AHSp-$\ell_\infty ^4$ are precisely those spaces $Y$ such that   $(\ell_\infty^4,Y)$ has the Bishop-Phelps-Bollob\'as property.     We also provide classes of Banach spaces   satisfying this condition. For instance, finite-dimensional spaces, uniformly  convex spaces, $C_0(L)$ and $L_1 (\mu)$ satisfy  AHSp-$\ell_\infty ^4 $.
\end{abstract}

  \maketitle

        \baselineskip=.65cm

   \section{Introduction}

Bishop-Phelps theorem \cite{BP} states that every continuous linear functional on a Banach space can be approximated (in norm) by norm attaining functionals.   Bollob{\'a}s proved a   \com{quantitative version} of that result \cite{Bol}. In order to state such result,  we denote by  $B_X$, $S_X$ and $X^*$   the closed unit ball,  the unit sphere
and the topological dual of a Banach space $X$, respectively.  If $X$ and $Y$ are  both real or  both complex Banach spaces, $L(X,Y)$ denotes the space of (bounded linear) operators from $X$ to $Y$, endowed with its usual operator norm.

%{\bf Notation: $B_X, S_X, L(X,Y)$}

\vskip3mm

{\it Bishop-Phelps-Bollob{\'a}s Theorem} (see \cite[Theorem 16.1]{BoDu}, or \cite[Corollary 2.4]{CKMMR}). Let $X$ be a Banach space and $0< \varepsilon < 1$. Given $x \in B_X$  and $x^* \in S_{X^*}$ with $\vert 1- x^* (x) \vert < \frac{\varepsilon ^2 }{4}$, there are elements $y \in S_X$ and $y^* \in S_{X^*}$  such that $y^* (y)=1$, $\Vert y-x \Vert < \varepsilon$ and $\Vert y^* - x^* \Vert < \varepsilon $.

A lot of attention has been  devoted to extending Bishop-Phelps theorem  to operators and
\linebreak[4]
 interesting results have been  obtained about that topic.  For instance, we mention the remarkable results by 
\linebreak[4]
Lindenstrauss \cite{Lin}, Bourgain \cite{Bou} and Gowers \cite{Gow}. In 2008  the study of extensions of
\linebreak[4]
 Bishop-Phelps-Bollob{\'a}s theorem to operators was initiated by Acosta, Aron, Garc{\'i}a and Maestre \cite{AAGM1}.  In order to state some of these extensions it will be convenient to  recall  the following notion.

\begin{definition} \label{def-BPBp}
	(\cite[Definition 1.1]{AAGM1}).  Let $X$  and $Y$ be either real or complex Banach spaces. The pair $(X,Y )$ is said to have the Bishop-Phelps-Bollob{\'a}s property for operators (BPBp) if for every $  0 < \varepsilon  < 1 $  there exists $ 0< \eta (\varepsilon) < \varepsilon $ such that for every $T\in S_{L(X,Y)}$, if $x_0 \in S_X$ satisfies $ \Vert T (x_0) \Vert > 1 - \eta (\varepsilon)$, then 
	there exist an element $u_0 \in S_X$  and an operator $S \in S_{L(X,Y )}$ satisfying the following conditions
	$$
	\Vert S (u_0) \Vert =1, \sem \Vert u_0- x_0 \Vert < \varepsilon \seg \text{and} 
	\sem \Vert S-T \Vert < \varepsilon.
	$$
\end{definition}

Acosta, Aron, Garc{\'i}a and Maestre  showed that the pair $(X,Y )$ has the BPBp whenever $X$ and $Y$ are finite-dimensional spaces \cite[Proposition 2.4]{AAGM1}. They also proved that 
the pair $(X,Y)$ has the BPBp  in case that $Y$ has a certain isometric property
(called property $\beta $ of Lindenstrauss), for  every Banach space $X$
\linebreak[4]
 \cite[Theorem 2.2]{AAGM1}.  For instance, the spaces $c_0$ and $\ell_\infty$ have such property. In case that  the domain is $\ell_1$ they obtained a characterization of the Banach spaces $Y$ such that $(\ell_1,Y)$  has the BPBp \cite[Theorem 4.1]{AAGM1}.   There  was also  proved that for  several classical spaces $Y$,  and for any positive measure $\mu$, the pair $(L_1(\mu), Y)$ have the BPBp, for instance when $Y=L_\infty (\nu)$ or  $Y=L_1 (\nu)$  (see  \cite{CK2}, \cite{ACGM}  and \cite{CKLM}). For those results the proofs are involved  and interesting. Even so,  as far as we know, there are no  characterization of the spaces $Y$ such that the pair $(L_1([0,1]),Y)$  has the BPBp.

%\end{document}
However, the case of $X = c_0$ is quite different from the case $X=\ell_1$ and  seems to be much more difficult.  Now we will list  results about this topic where the domain is a space $C_0(L)$  ($L$ is  a locally compact Hausdorff  space).  It was shown that $(\ell_{\infty}^n, Y )$  has the BPBp for every positive integer $n$ whenever $Y$ is uniformly convex \cite[Theorem 5.2]{AAGM1}. In fact Kim proved that  $(c_0,Y)$  also has this property under the same assumption \cite[Corollary 2.6]{KiI}.  Aron, Cascales and   Kozhushkina showed that $(X, C_0(L))$ has the BPBp if $X$ is Asplund \cite[Corollary 2.6]{ACK} (see also \cite{CGK} for 
an extension of this result).  As a consequence, the pair $(c_0, C_0(L))$ has the Bishop-Phelps-Bollobás property for operators. In the real case, it is also known that the pair $(C(K), C(S))$
also satisfies the BPBp, for any compact Hausdorff spaces $K$ and $S$ \cite[Theorem 2.5]{ABCCetc}. Kim,  Lee and Lin proved that the pair  $(L_\infty(\mu), Y)$  has
BPBp  for  every positive measure $\mu$,  whenever $Y$ is uniformly convex \cite[Theorem 5]{KLL}. Kim and Lee  extended such result for the pair $(C(K),Y)$ ($K$ is compact Hausdorff spaces)  \cite[Theorem 2.2]{KL}. In the complex case, Acosta proved that the pair $(C_0(L), Y)$  has the BPBp  for every locally  compact Hausdorff space $L$ and any  $\C$-uniformly convex space $Y$  \cite[Theorem 2.4]{AcB}. In the complex case $L_1 (\mu)$ is   $\C$-uniformly convex space, so the previous result can be applied.
There are some other sufficient conditions on a Banach space $Y$  in order that the pair $(c_0,Y)$ satisfies the BPBp (see for instance \cite[Theorem 2.4]{AGKM}).

% For more results where the domain space is some space C(K) see %also [?] and [?]. 

%Since c0 is Asplund it follows that (c0,C0(L)) has the BPBp .  

% ( SUMAS $c_0$ en partida. Y el trabajo de PRIMS para  $N=3$. LUEGO HAY QUE DEFENDER %ESTE TRABAJO Y DAR UN LISTADO DE RESULTADOS: caracterizacion y ejemplos.

However until now there is no characterization of the spaces Y such that $(c_0,Y )$ has the BPBp. Indeed in the real case it is not known whether or not the  pair  $(c_0, \ell_1)$  has the BPBp. As a consequence of \cite[Theorem 2.1]{ACKLM}, in case that   the pair   $(c_0, \ell_1)$ has the  BPBp, then the pairs  $(\ell_\infty ^n, \ell_1)$  satisfy the BPBp \com{uniformly} for every $n$. For this reason  in this paper we approach this problem by using in the domain appropriate finite-dimensional spaces.

Notice that for dimension $2$,  since  $\ell_\infty ^2 $  is isometrically isomorphic to $\ell_{1}^2$, as a consequence of 
\linebreak[4]
 \cite[Theorem 4.1]{AAGM1},  it is known   that $(\ell_\infty ^2,Y)$ has the BPBp if and only if $Y$ has the approximate
 \linebreak[4]
  hyperplane series property for convex combinations of two elements. A characterization of the spaces $Y$ such that $(\ell_\infty ^3,Y)$ has the BPBp was shown in \cite[Theorem 2.9]{ABGKM2}. As a consequence of that result, classical Banach spaces satisfying the  previous property were provided. The goal of this paper is  to extend the  above mentioned  results for the pair  $(\ell_\infty ^4,Y)$.

%More precisely,   we  will introduce AHSp-$\ell_{\infty}^4 $   %and then we will show that a  Banach space $Y$ satisfies that  %$(\ell_\infty ^4,X )$ has the BPBp if and only if  $Y$ has  %AHSp-$\ell_{\infty} ^4$.

Now we briefly describe the content of the paper. In section 2 we introduce a geometric property on a Banach space $Y$, namely the AHSp-$\ell_{\infty} ^4$ (see Definition \ref{def-AHSPL4}). We also provide several reformulations of such property in Proposition \ref{pro-char}. That result is essential to prove in section 3 that the pair $(\ell_\infty ^4,Y)$ has the BPBp for operators  if and only if $Y$ has the AHSp-$\ell_{\infty} ^4$  (see Theorem \ref{teo-char}).
In section 4 we provide examples of classical  spaces satisfying the AHSp-$\ell_{\infty} ^4$.  For instance,  we  check  that  finite-dimensional spaces and uniformly convex spaces  have this property. It is also satisfied that  $C_0(L,Y)$  has the AHSp-$\ell_{\infty} ^4$ whenever $Y$ has the same property, for any locally compact Hausdorff space $L$  (Proposition \ref{pro-C0-Y}). Lastly we prove that $\ell_1$ has the AHSp-$\ell_{\infty} ^4$ (Proposition \ref{prop-ell1}). The proof of that result requires some effort.  As a consequence of the result for $\ell_1$, we obtain that 
$(\ell_\infty ^4, L_1(\mu))$ has the Bishop-Phelps-Bollob{\'a}s property for any  positive measure $\mu$.

 Throughout this paper  we follow the spirit   of the results obtained in \cite{ABGKM2} for $\ell_\infty ^3$, but the proofs are  much more  complicated. We provide a few arguments for that. 
  One reason is that for $\ell_{\infty} ^3$ any subset of three  extreme points contained in the same face of the unit ball  can be applied to any other subset of the same kind  by using an appropriate linear isometry.
%in the same face of the unit ball can be applied in any other three extreme points  %satisfying the same conditions  
For dimension $4$ the previous statement is not satisfied.
Because of  that the image under an operator of any three  extreme points in  the same face of $B_{\ell_\infty^ 3}$  is a good choice  in order to identify an operator whose domain is $\ell_\infty^3$.  In  case   that the domain of the operator is $\ell_\infty ^4$ we had to find   an appropriate choice of the basis in order to identify   such operator   and to compute its norm (Proposition  \ref{prop-ident}).  Another reason is that due to the bigger amount of extreme points in the unit ball of $\ell_\infty^4 $  the description used to identify operators whose domain is $\ell_\infty ^4$ is  more involved.  As a consequence, the property of  $Y$ equivalent to the fact that 
$(\ell_\infty ^4,Y)$ has the BPBp is more complicated. That reason also makes to provide examples more difficult.

Throughout the paper we consider {\it real } normed spaces.  By $\ell _ \infty ^4$ we denote the space $\R^4$, endowed with the norm given by $\Vert x \Vert = \max \{ \vert x_i \vert : i \le 4 \}$.

\section{The approximate hyperplane sum property for $\ell_\infty ^4$}

 In this section we identify the unit ball of $L(\ell_\infty ^4, Y)$ with a certain family of elements in $Y^4$ called $M_Y^4$ (Proposition \ref{prop-ident}). We also introduce an intrinsic  property on  a Banach space $Y$, namely the AHSp-$\ell_\infty ^4$ and show a characterization of that property  (see Proposition \ref{pro-char}).

\begin{notation}
		\label{not-M4Y}
	If $Y$ is a Banach space, we will denote by
	$$
	M^4_Y= \{ (y_i)_{i \le 4}  \in (B_Y)^4:  y_{i_1}-y_{i_2}+y_{i_3} \in B_Y, \forall   1 \le i_1 < i_2 < i_3 \le 4 \}.
	$$
	\end{notation}

\begin{remark}
It is clear that $(y_i)_{i \le 4} \in M^4_Y $  if  $(-y_i)_{i \le 4} \in M^4_Y $.
\end{remark}

The following notion is analogous to the AHSp-$\ell_\infty ^3$  that was used to characterize those spaces $Y$ such that the pair $(\ell_\infty ^3, Y)$ has the BPBp for operators  (see \cite[Definition 2.1]{ABGKM2}).

\begin{definition}
	\label{def-AHSPL4}
	A Banach space $Y$  has  the  {\it  approximate  hyperplane  sum property for       $\ell_\infty^4$} 
	\linebreak[4]
	(AHSp-$\ell_\infty^4$)  if for every $0 < \varepsilon < 1$ there is $ 0< \gamma (\varepsilon) < \varepsilon $ satisfying the following condition
	
	For every $ (y_i)_{i \le 4}  \in M_Y^4$,  if there exist a nonempty subset  $A$ of $\{1,2,3,4\}$ and $y^* \in
	S_{Y^*}  $ such that $y^* (y_i) > 1 - \gamma (\varepsilon) $ for  each $i \in A$, then there  exists an element $ (z_i)_{i \le 4}  \in M_Y^4$  satisfying $\Vert z_i - y_i \Vert < \varepsilon $ for
	every $i \le 4$ and $\Vert \sum_{i \in A } z_i \Vert = \vert A\vert $.
\end{definition}

\vskip3mm

\begin{remark}
\label{re-4-3}
We recall that  $Y$ has the AHSp-$\ell_\infty ^3$
 if for every $\varepsilon > 0$ there is $\gamma
 > 0$ satisfying the following condition:
 
 For a subset $\{y_i: i\le 3\} \subset  B_Y $ with $\Vert y_1 +
 y_2 - y _ 3 \Vert \le 1$ , if there exist a nonempty subset  $A$ of $\{1,2,3\}$ and $y^* \in
 S_{Y^*}  $ such that $y^* (y_i) > 1 - \gamma $ for  every $i \in A$, then there  exists
 $\{z_i: i \le 3\} \subset B_Y$  with $\Vert z_1 +
 z_2 - z_3 \Vert \le 1$ satisfying $\Vert z_i - y_i \Vert < \varepsilon $ for
 every $i \le 3$ and $\Vert \sum_{i \in A } z_i \Vert = \vert A \vert $.
 
 \vskip3mm
 
 If $(y_i)_{i \le 3}$ satisfies the assumption in the definition of the AHSp-$\ell_\infty ^3$, it is immediate that 
 \linebreak[4]
 $(y_1, y_3, y_2, y_2) \in M_{Y}^4$. This is the key idea to check that AHSp-$\ell_\infty ^4$ implies AHSp-$\ell_\infty ^3$. 
\end{remark}

 Now we state some basic but useful results.

%Now  we will state some basic results that %will be useful  later in order  to provide %examples of spaces with the %AHSp-$\ell_\infty^4$.

\begin{lemma}
	\label{le-t-sets}
	If $\{ y_i: 1 \le i \le 4\} \subset B_Y$, the following conditions are equivalent
	\begin{enumerate}
		\item $( y_1, y_2, y_3, y_4 ) \in M_Y^4 $
		\item $(  y_2, y_3, y_4, -y_1 ) \in M_Y^4 $
		\item $( y_3, y_4, -y_1, -y_2 ) \in M_Y^4 $
		\item $( y_4,-y_1, -y_2, -y_3)  \in M_Y^4 $
	\end{enumerate}
\end{lemma}
\begin{proof}
	Statement (1) is satisfied exactly when the following elements belong to  $B_Y$
	\begin{equation}
		\label{cond-1}
		y_1-y_2+y_3, \sep  y_1-y_2+y_4, \sep  y_1-y_3+y_4, \sep  y_2-y_3+y_4.
	\end{equation}
	On the other hand, condition  (2) means that  each of the following elements belongs to $B_Y$
	\begin{equation}
		\label{cond-2}
		y_2-y_3+y_4, \sep  y_2-y_3-y_1, \sep  y_2-y_4-y_1, \sep  y_3-y_4-y_1.
	\end{equation}
	As a consequence  conditions (1) and (2)  are equivalent. By applying this fact we obtain that (2)
	and  (3) are equivalent. Again by the same argument (3) and (4) are equivalent.
\end{proof}

Next result shows that the condition stated in  Definition \ref{def-AHSPL4} is trivially satisfied in case that the set $A$  contains a unique element.

\begin{proposition}
	\label{pro-AHSP-1-element}
	Let $Y$ be a Banach space. Let   $ 0 < \varepsilon < 1$ and $ (y_i)_{i \le 4} \in M_Y^4$.
	Assume that 
	\linebreak[4]
	$A\subset \{1,2,3,4\}$  contains only one  element  and it is satisfied that
	$$
	\Vert y_j\Vert  > 1 - \frac{\varepsilon}{6}, \sem j \in A.
	$$
	Then there is an element $(z_i)_{i \le 4} \in M_Y^4$  such that
	$$
	\Vert z_j \Vert =1 \sep \text{for} \sep j \in A \sem \text{and} \sem  \Vert z_i - y_i \Vert < \varepsilon, \sep \forall 1 \le i \le 4 .
	$$
\end{proposition}
\begin{proof}
	Let  $ 0 < \varepsilon < 1$ and $(y_i )_{i \le 4} \in M_Y^4$.
	In view of Lemma \ref{le-t-sets}  it suffices to show the statement  in case that $A=\{1\}$.
	
	Assume that the element $y_1$ satisfies that
	$$
	\Vert y_1\Vert  > 1 - \frac{\varepsilon}{6} > 0.
	$$
	We define the following real numbers
	$$
	a=\frac{\varepsilon}{3}, \sem  a_2= \frac{\varepsilon}{3}, \sem   a_3= \frac{\varepsilon}{6},  \sem  a_4=-\frac{\varepsilon}{6},
	$$
	and the elements in $Y$
	$$
	z_1=\frac{y_1}{\Vert y_1\Vert }  \sem \text{and}\sem z_i=(1-a)y_i + a_iz_1, \sep  i\in \{2,3,4\}.
	$$
	It is  trivially satisfied that $z_1 \in S_Y$. We also have that
	\begin{equation}
	\label{z1-y1}
	\Vert  z_1 - y_1\Vert = 1- \Vert y_1\Vert  < \frac{\varepsilon}{6} < \varepsilon.
	\end{equation}
	In case that $2 \le  i\le 4$  we obtain that
	$$
	\Vert z_i\Vert  \leq (1-a) + |a_i| \leq 1- \frac{\varepsilon}{3} + \frac{\varepsilon}{3} = 1
	$$
	and
	$$
	\Vert z_i-y_i\Vert =\Vert -ay_i + a_iz_1\Vert  \leq a +|a_i|\leq  \frac{2 }{3} \varepsilon  < \varepsilon.
	$$
	It remains to show that for every $1 \le i_1 < i_2 < i_3 \le 4$ it is satisfied that
	$$
	\Vert z_{i_1}-  z_{i_2}+ z_{i_3} \Vert \le 1.
	$$
	In case that $\{i_1, i_2, i_3\}=\{2,3,4\}$ we obtain
	$$
	\Vert z_2 - z_3 + z_4\Vert  = \Vert  (1-a)(y_2-y_3+y_4) + (a_2-a_3+a_4)z_1 \Vert  =  \Vert  (1-a)(y_2-y_3+y_4)  \Vert \leq 1-a < 1.
	$$
	Otherwise $i_1=1$. Notice that for every $1 < i_2 < i_3 \le 4$ we have that
	$$
	\vert a - a_{i_2}+a_{i_3} \vert \le \frac{\varepsilon}{6}.
	$$
	Hence
	\begin{align*}
	\Vert  z_1-z_{i_2}+z_{i_3}\Vert  &=  \Vert  (1-a)(z_1-y_{i_2}+y_{i_3}) +(a - a_{i_2}+a_{i_3})z_1\Vert  \\
	& \leq
	(1-a)(1+ \Vert z_1-y_1\Vert )+ |a - a_{i_2}+a_{i_3}| \\
	&  < \Bigl( 1 - \frac{ \varepsilon}{3}\Bigr)  \Bigl(1 +\frac{\varepsilon}{6} \Bigr) + \frac{\varepsilon}{6} \seg \sem \text{(by \eqref{z1-y1})}\\
	& =   1 -   \frac{  \varepsilon}{3} +  \frac{  \varepsilon}{6}  -   \frac{  \varepsilon^2 }{18} +  \frac{  \varepsilon}{6} \\
	&  = 1 -    \frac{  \varepsilon^2 }{18}\\
	&
	< 1.
	\end{align*}
	We checked that $(z_i)_{i \le 4} \in M_Y^4$ and  it satisfies all the required conditions.
\end{proof}

%It is  clear that $ z \in A \subset Y$ is a $\mathcal{T}$-set if %$-A$ is a
%$\mathcal{T}$-set. Since the notion of $\mathcal{T}$-set depends %of the order of
%the elements sometimes we will denote by $(y_1, y_2, y_3, y_4)$  %the
%$\mathcal{T}$-set $\{y_i: 1 \le i \le 4 \}$.

\begin{notation}
	\label{not-E1-vi}
In what follows we will denote by $E_1$ the subset of $B_{\ell^4_\infty}$ given by
$$
E_1 = \{ x \in \ell_\infty ^4: 1= x(1)= \Vert x \Vert \}.
$$
In the sequel it will be convenient to use the following notation for the following  elements  in $E_1$
$$
v_1 =(1,1,1,1), \sep  v_2 =(1,-1,1,1), \sep v_3 =(1,-1,-1,1), \sep v_4
=(1,-1,-1,-1),
$$
$$
v_5 =(1,1,-1,1), \sep  v_6=(1,1,-1,-1), \sep v_7=(1,1,1,-1), \sep v_8 =(1,-1,1,-1),
$$
By  $\mathcal{B}$ we will denote the set given by $\mathcal{B}=\{ v_i: 1 \le i \le
4\}$.
\end{notation}

The proofs  of next assertions  are straightforward. For the first one it suffices  to check that   every coordinate of an  extreme point of $E_1$   belongs to $\{1,-1\}$.  %In fact we have the following elementary result. 

\begin{lemma}
	\label{le-ext-E1}
	It is satisfied that
	$$
	\ext \bigl(E_1 \bigr)= \{ v_i : i \le 8 \}.
	$$
\end{lemma}

\begin{lemma}
	\label{le-basis}
	The set $\mathcal{B}$ is a basis of $\R^4   $ contained in $S_{ \ell_{\infty}^4 }$.
	Moreover the functionals given by
	$$
	v_{1}^* (x)= \frac{ x(1) +  x(2)}{2},  \sep v_{i}^* (x)= \frac{ x(i+1) -  x(i)}{2},
	\  i=2,3 \  \text{and} \sep v_{4}^* (x)= \frac{ x(1) -  x(4)}{2},
	$$
	are elements in $S_{(\ell_{\infty}^4 )^*}$ that are the biorthogonal functionals of
	the basis $\mathcal{B}$.
	Hence each   $x$ in $\R ^4$ can be expressed as
	$$
	x= \frac{ x(1) +  x(2)}{2} v_1 + \sum _{i=2}^3 \frac{x(i+1)- x(i)}{2} v_i +
	\frac{x(1)- x(4)}{2} v_4.
	$$
\end{lemma}

As a consequence of the last assertion in the previous result we obtain the following.

\begin{remark}
	\label{re-vi-bad-vi-good}
The following equalities are satisfied
$$
v_5 =  v_1-v_2 + v_3 , \sem v_6= v_1 - v_2 + v_4 ,
$$
and 
$$
\sep v_7 = v_1 - v_3 + v_4 , \sem v_8 = v_2 - v_3 + v_4 .
$$
Hence $(v_i)_{ i \le 4} \in M_{Y}^4 $ for $Y= \ell _\infty ^4$.
\end{remark}

The next result shows the connection between operators  from $\ell_\infty ^4 $ to $Y$  and the set $M^4_Y$. 

\begin{proposition}
	\label{prop-ident}
	Every element $T \in L(\ell_\infty ^4, Y)$  satisfies 
	$$
	\Vert T \Vert = \max \{ \bigl\Vert T (v_{i_1} - v_{i_2} + v_{i_3})\bigr\Vert:
	1 \le i_1 \le i_2\le i_3 \le 4\}.
	$$
So  the  mapping given by  $\Phi (T)= ( T(v_i))_{i \le 4} $
identifies  $ B_{L(\ell_\infty ^4, Y)}$  with $M_{Y}^4$.
%V1,  which identifies   $L(\ell_\infty ^4, Y)$ with $Y^4$,  maps   $ B_{L(\ell_\infty %^4, Y)}$  to  $M_{Y}^4$.
	%	V2 Hence the mapping given by  $\Phi (T)= ( T(v_i))_{i \le 4} $ is an identification %from  $L(\ell_\infty ^4, Y)$ to $Y^4$.    As a consequence, $ B_{L(\ell_\infty ^4, %Y)}$ is identified with $M_{Y}^4$ under $\Phi$.
\end{proposition}
\begin{proof}
	In view of  Lemma \ref{le-basis}   the set $\mathcal{B}=\{ v_i : i \le 4\}$ is a basis of $\R^4$,  so every  operator $T$ from $ \ell_\infty ^4$ to $Y$  is determined by the element $( T(v_i))_{ i \le 4}$.
	
	It is clear $\ext \bigl( B_{\ell_\infty ^4}\bigr)  = \ext \bigl(  E_1 \bigr)  \cup \ext \bigl(-E_1\bigr)$.  Hence
\begin{eqnarray}
\label{ope-norm}
\nonumber
\Vert T \Vert&=& \max \{ \Vert T(e) \Vert : e \in \ext \bigl( B_{\ell_\infty ^4 }\bigr) \} \\
&=& 	\max \{ \Vert T(e) \Vert : e \in \ext (E_1) \} \\
\nonumber
&=& \max \{ \Vert T(v_i) \Vert : i \le 8 \},
\end{eqnarray}
	where we used Lemma \ref{le-ext-E1}.
	
	In view of Remark \ref{re-vi-bad-vi-good} we have that
	$$
	\{ v_i : 5 \le i \le 8\} = \{  v_ { i_1} -  v_ { i_2} +v_ { i_3}  :   1 \le i_1 < i_2 < i_3  \le 4 \}.
	$$
 From  \eqref{ope-norm} and the previous equality we obtain   that
$$
\Vert T  \Vert =  \max \{ \bigl\Vert T (v_{i_1} -v_{i_2} + v_{i_3})\bigr\Vert:
	1 \le i_1 \le i_2\le i_3 \le 4\}.
	$$
It  follows  that  $T \in B_{L(\ell_\infty ^4, Y)}$ if and only if $\Phi(T) $ is an element in $M_Y ^4$.
\end{proof}

We  recall that a subset   $B\subset B_{Y^{*}}$ is $1$-\textit{norming} if  $\|y\|=\sup \{\vert y^{*}(y) \vert : y^{*}\in B\}$ for each  $y\in Y$. Now we provide a characterization  of  AHSp-$\ell_\infty^4$ which will be used in the following sections.

\begin{proposition}
    \label{pro-char}
    Let $Y$ be a    Banach space. The following conditions are equivalent:
    \begin{enumerate}
        \item[1)]  $Y$  has the approximate  hyperplane  sum property for       $\ell_\infty^4$.
       %satisfies AHSp-$\ell_\infty^4$.
        \item[2)]  There is a $1$-norming subset $B \subset S_{Y^*}$ such that the condition
        stated in Definition \ref{def-AHSPL4}   is satisfied for every $y^*\in B$.
        \item[3)]
        For every $0< \varepsilon <1$ there exists  $0 < \nu ( \varepsilon)<  \varepsilon $ such that for every  element $(y_i)_{i \le 4} \in M_Y^4$  and each convex combination  $\sum_{i=1}^4 \alpha_i y_i $ satisfying
        $$
        \Bigl \Vert \sum_{i=1}^4 \alpha_i y_i\Bigr \Vert  > 1-\nu (\varepsilon ),
        $$
        there exist a  set $A\subset \{ 1,2,3,4\}$ and an element $(z_i)_{i \le 4} \in M_Y ^4$  such that
        \begin{enumerate}
            \item[i)] $\sum_{i\in A}\alpha_i>1-\varepsilon,$
            \item[ii)]  $\|z_i-y_i\|<\varepsilon  \ \ \ \text{for each } i \le 4,$
            \item [iii)] $\Vert \sum _{i \in A } z_i \Vert = \vert A \vert$.
            %             $y^\ast(z_i)=1 \text{\ for all } i\in A,$
        \end{enumerate}
    \end{enumerate}
 Moreover, if $\rho  $ is the function satisfying condition 2), then   condition 3) is also satisfied with  the function   
 $\nu = \rho ^2$. In case that 3) is satisfied with a  function $\nu $,   $Y$  has the AHSp-$\ell_\infty ^4$ with the function $\gamma(\varepsilon)=\nu(\frac{\varepsilon}{4}) $.

%Moreover, if $\gamma $ is the function %satisfying condition 1) (see Definition %\ref{def-AHSPL4}), condition 2) is also %satisfied with $\gamma $. In case that 2) is %satisfied for the function $\rho$, the %function   
%$\nu = \rho ^2$ satisfies 3). If    a %function  $\nu  $ satisfies condition 3),  %the function $\gamma $  given by  %$\gamma(\varepsilon)=\nu(\frac{\varepsilon}{4%}) $  satisfies 1). HOLDS
\end{proposition}
\begin{proof}
    Clearly 1) implies 2).

    \noindent
    2) $\Rightarrow$ 3)
    \newline
    Assume that $Y$ satisfies condition 2).  For each  $0<\varepsilon < 1$  let be $ \rho(\varepsilon)<\varepsilon$ the positive real number  satisfying  Definition \ref{def-AHSPL4} for every element $y^{*}\in B$. We take  $\nu(\varepsilon)=(\rho(\varepsilon))^2$.
    %HE PUESTO EL CUADRADO EN LUGAR DEL CUBO EN LA ELECCION DE $\eta$

    Let   $ (y_i)_{i \le 4} \in M_Y^4$  and assume that the convex combination  $\sum_{i=1}^{4} \alpha_i y_i$   satisfies   $\bigl\Vert  \sum_{i=1}^{4} \alpha_iy_i \bigr \Vert  > 1 - \nu(\varepsilon) $. Since $(-y_i)_{i \le 4} \in M_Y^4$, by  using $(-y_i)  $ instead of $(y_i)$, if needed,  since $B$ is a $1$-norming set, there is  $y^{*}\in B$ such that
    $$
   y^{*}\biggl( \sum_{i=1}^{4} \alpha_i y_i \biggr) =    \sum_{i=1}^{4} \alpha_iy^{*}(y_i) >  1 - \nu(\varepsilon) = 1 - \rho (\varepsilon) ^2 .
    $$
    By \cite[Lemma 3.3]{AAGM1} the set $A$ given by  $A:=\{ i\le 4 : y^*(y_i)> 1-\rho (\varepsilon) \}$
    satisfies
    %\}$, A es no vacio y por \cite[Lema 3.3]{Original} se cumple que
    $$
    \sum_{i\in A}\alpha_i \geq 1 - \frac{\nu(\varepsilon)}{\rho (\varepsilon)}> 1- \varepsilon.
    $$
    By assumption there is an element $(z_i )_{i \le 4}  \in M_Y^4$  such that  $\| z_i - y_i\| < \varepsilon$ for each  $i \leq 4$ and  $\|\sum_{i \in A} z_i \| = |A|$.

    \noindent
    3) $\Rightarrow$ 1)
    \newline
Now we assume that  $Y$ satisfies condition 3).
    Let be $0<\varepsilon < 1$ and  $\nu(\varepsilon)$ the positive real number satisfying the assumption. We will show that  $\gamma(\varepsilon)=\nu(\frac{\varepsilon}{4})$ satisfies  Definition \ref{def-AHSPL4}.

    Let $(y_i)_{i \le 4} \in M_Y^4$  and assume that for some   nonempty set  $A \subset \{ 1,2,3,4 \}$ and  $y^{*} \in S_{Y^{*}}$ it is satisfied that  $y^{*}(y_i)>1-\gamma(\varepsilon)$ for each  $i\in A$.   We define the following nonnegative real numbers
    $$
    \alpha_i =
    \begin{cases}
    \frac{1}{|A|}            & \mbox{if} \sep i  \in A   \\
    0  & \mbox{if} \sep  i \in \{ 1,2,3,4\} \backslash  A.
    \end{cases}
    $$
    Clearly $\sum _{i=1}^ 4 \alpha _i =1$ and  we also have that
    $$
    \Big\| \sum_{i =1}^{4}\alpha_iy_i \Big\| = \frac{\Big\| \sum_{i \in A}y_i \Big\| }{|A|}\geq \frac{y^{*}\Big( \sum_{i \in A}y_i \Big )}{|A|}> 1-\nu\Big(\frac{\varepsilon}{4}\Big).
    $$

    By assumption there is a set   $C\subset \{ 1,2,3,4\}$ and $(z_i)_{i\le 4} \in  M_Y^4$  such that
    \begin{enumerate}
        \item[i)] $\sum_{i\in C}\alpha_i>1-\frac{\varepsilon}{4}> \frac{3}{4},$
        \item[ii)]  $\|z_i-y_i\|<\frac{\varepsilon}{4}  \sep \text{for each  }\sep  i \le 4,$
        \item [iii)]  $ \bigl \Vert \sum _{ i \in C} z_i \bigr\Vert  = \vert C \vert $.
    \end{enumerate}

    If  $A\subset C$ then the proof will be finished since condition iii)  implies that  $\bigl \Vert \sum _{ i \in A} z_i \bigr\Vert  = \vert A \vert $.  In case that there is some $i_0 \in A\setminus C$,  we put  $B= \{i \le 4: i \ne i_0\}$ 
    and so 
    $$
    1- \frac{1}{\vert A \vert} =  \frac{|A|-1}{|A|}=\sum_{ i \in B} \alpha_i\geq \sum_{i\in C} \alpha_i >\frac{3}{4}.
    $$
    Then  $\vert A \vert >  4$, which is a contradiction. Hence $A\subset C$ and  we proved  that $Y$ has the AHSp-$\ell_\infty^4$.

    Of course, if $\gamma $ satisfies Definition \ref{def-AHSPL4}, then $\rho= \gamma$ also  satisfies condition 2). 
    In case that 2) is satisfied with the function $\rho $, we showed that  condition 3) is also satisfied with the function  $\nu= \rho ^2$. Lastly if we assume that 3) is true with a function $\nu$ we know  that $Y$ has the AHSp-$\ell_\infty ^4$  with  the function $ \varepsilon \mapsto \nu \bigl( \frac{\varepsilon}{4}\bigr)$ because of the proof of 3) $\Rightarrow $ 1).
\end{proof}

%{\bf  Next  section should be deleted (only the title)  and we %have put the results in the right place }

%\section{Properties of the $\mathcal{T}$-sets}

\section{A characterization of the  spaces $Y$ such that the pair $( \ell_\infty^4, Y)$ has the Bishop-Phelps-Bollob{\'a}s property}

In  this section  we show  that the pair $(\ell_\infty ^4,Y)$ has the Bishop-Phelps-Bollob{\'a}s property for operators if and only if $Y$ has the AHSp-$\ell_\infty ^4$. Some technical results will make the proof easier. As usual, we denote by $\co (A)$ the convex hull of a subset $A$ of a linear space.

\begin{lemma}
    \label{le-faces}
    If $x \in E_1$ and $x(i)\le x(i+1)$ for $2\le i \le 3$, then $x \in \co \, \mathcal{(B)}$.
 It is satisfied that  $E_1= \cup_{k=1} ^6 \co (\{ v_i : i \in A_k\})$, where 
$$
A_1=     \{1,2,3,4\},  \sep A_2= \{1,3,4,5\}, \sep A_3=  \{1, 4, 6, 7\},
$$
$$
A_4=  \{1, 2, 4, 8\}, \sep A_5=   \{ 1,4,5,6\}\sep  \text{\rm and} \sep A_6= \{1,4,7,8\}.
$$
Indeed  for every $1 \le k \le 6$,  $\{ v_i: i \in A_k\}$ is the image under an   appropriate  linear isometry on $\ell_\infty ^4$ of $\mathcal {B}$.
\end{lemma}
\begin{proof}
If $x \in E_1$, by Lemma     \ref{le-basis} we know that
\begin{equation}
\label{x-convex}
x= \frac{ 1 +  x(2)}{2} v_1 + \sum _{i=2}^3 \frac{x(i+1)- x(i)}{2} v_i +  \frac{1-
x(4)}{2} v_4.
\end{equation}
In case that $x(2)\le x(3 ) \le x(4)$ notice that    $x$ is expressed in
\eqref{x-convex}  as a convex combination of $\{ v_i: 1 \le i \le 4\}$.

For each  permutation $\sigma $ of  $\{ 2,3,4\}$ we define  the linear  isometry
$T_\sigma $ on $\ell_{\infty} ^4$ given by
$$
T_\sigma (x)=\bigl( x(1),x(\sigma (2)), x(\sigma (3)), x(\sigma (4)) \bigr) \seg (x
\in \R^4).
$$
Notice that $T_\sigma $ preserves $E_1$.

If $x \in E_1$ and  $\sigma$ is a permutation of $\{2,3,4\}$   is such that  $ x
(\sigma (2)) \le x(\sigma (3)) \le x(\sigma (4)) $  we know that the element
$T_\sigma (x)$  can be expressed as a convex combination of $\{ v_i: 1 \le i
\le 4\}$. Hence  $x=T_{\sigma^-1 } (T_ \sigma (x))$ can be expressed as a convex
combination of $\{ T_{\sigma^-1 } ( v_i) : 1 \le i \le 4 \}$. So it suffices to
compute the images by $T_\sigma $  of $\{ v_i: 1 \le i \le 4\}$, where $\sigma $ is
any permutation of $\{ 2,3,4\}$. Notice that the elements $ v_1$ and $v_4$ are
invariant by all these isometries.  So it suffices to evaluate the image of the
elements $\{v_2, v_3\} $.

We include the results in the following table, where we denote by $I$ the identity
and $\tau_{i,j} $ the transposition of the elements $i$ and $j$ on $ \{2,3,4\}$

\vskip5mm

$$
\begin{array}{|r|l|}
    \hline
\sigma  &  T_\sigma (\{v_2, v_3\})\\
    \hline
I &  \{v_2, v_3\}\\
\hline
\tau_{2,3}  & \{v_3 , v_5 \} \\
    \hline
\tau_{2,4}  & \{ v_6, v_7 \} \\
    \hline
\tau_{3,4}  &  \{ v_2,v_8 \}\\
    \hline
\tau_{2,3} \circ  \tau_{3,4}  & \{ v_5,v_6 \} \\
    \hline
\tau_{2,3} \circ  \tau_{2,4}  & \{v_7 , v_8 \} \\
    \hline
\end{array}
$$

\vskip5mm

Then we obtained that $E_1= \cup_{i=1} ^6 \co \{ v_j: j \in A_i \}$, where the sets
$A_i$  are given by
$$
A_1= \{ 1,2,3,4\}, \sep  A_2=\{ 1,3,4,5\},  \sep A_3 =\{ 1,4,6,7\},
$$
$$
A_4 =\{ 1,2,4,8\}, \sem A_5= \{ 1,4,5,6\} \sep \text{and} \sep  A_6=\{ 1,4,7,8\}.
$$
\end{proof}

The next result gives  a procedure to change an element $u_0$ close to $\co \{ v_i:1 \le i \le 4\}$ by a new one satisfying more requirements.

\begin{lemma}
\label{le-close-face}
Assume that $0 < \varepsilon < \frac{1}{2}$, $x_0 \in \co \{ v_i:1 \le i \le 4\}$
and   $u_0 \in E_1$ satisfies that $\Vert u_0 - x_0 \Vert < \varepsilon$. Then
there is $v_0 \in E_1$ such that $\Vert v_0 - u_0 \Vert < 3 \varepsilon$ and such
that there is  a set $A \subset \{ i \in \N : i \le 8\}$  satisfying that   $u_0$
and  $v_0$ can be written as convex combinations as follows
$$
u_0= \sum _{i \in A} \beta _i v_i , \sep v_0= \sum _{i \in A, i \le 4} \gamma _i
v_i
$$
and also
$$
  \gamma _i > 0 \sep \text{\rm for some} \sep i \ \Rightarrow \ \beta _i > 0.
$$
\end{lemma}
\begin{proof}
By assumption  we know that $x_0$ can be written as  $x_0= \sum_{i=1}^{4} \alpha_i v_i$, where $\alpha _i \ge 0$ for $i \le 4$ and $\sum _{i=1}^4 \alpha _i=1$.

By Lemma \ref{le-faces} $u_0\in \cup_{i=1}^6
\co (\{ v_i: i \in A_k\})$. If  $u_0 \in \co 
\{ v_i: i \in A_1\}=\{ v_i: i \le 4\}$ then the element $v_0=u_0$
satisfies the statement.  
Otherwise  we can express  
$u_0= \sum _{i\in A} \beta _i v_i$, where $\beta _i \ge 0$ for each $i \in A$ and
$\sum _{i \in A} \beta _i =1$ and it  suffices to
prove the claim in the following cases:

\noindent
 Case a) $A=A_2 \cup A_5 = \{1,3,4,5,6\}$. Since the functional $v_{2} ^*  $ given by
$ v_{2}^* (x)=\frac{ x(3) - x(2)}{2}$ belongs to the unit ball of
$(\ell_{\infty}^4)^*$,  in view of Lemma \ref{le-basis} and Remark \ref{re-vi-bad-vi-good} we obtain that
\begin{equation}
\label{beta5-6}
\alpha _2 + \beta _5 + \beta _6  = v_{2} ^* (x_0- u_0) \le \Vert x_0 - u_0 \Vert <
\varepsilon.
\end{equation}
 Case b) Assume that $A=A_3= \{ 1,4,6,7\}$. The functional $v^* $ given by
$ v^* (x)=\frac{ x(4) - x(2)}{2}$ belongs to the unit ball of $(\ell_{\infty}^4)^*$ and
satisfies
$$
v^* (v_i)=0, \sep i=1,4   \sem v^* (v_i)=1, \sep i=2, 3,   \sem  \text{and}\sem
v^* (v_i)=-1, \sep i=6,7.
$$
As a consequence we have that
\begin{equation}
\label{beta6-7}
\alpha _2 + \alpha _3 +  \beta _6  + \beta _7 = v^* (x_0- u_0) \le \Vert x_0 -
u_0 \Vert < \varepsilon .
\end{equation}
Case c) Assume that $A = A_4 \cup A_6=\{ 1,2,4,7,8\}$.  By Lemma \ref{le-basis} and Remark \ref{re-vi-bad-vi-good}  we obtain that 
\begin{equation}
\label{beta7-8}
\alpha _3 + \beta _7 + \beta _8  = v_{3} ^* (x_0- u_0) \le \Vert x_0 - u_0 \Vert <
\varepsilon.
\end{equation}

In each of the above cases, we take  $B= A \cap  \{1,2,3, 4\}$. Notice that in view
of \eqref{beta5-6}, \eqref{beta6-7} and \eqref{beta7-8},  it is satisfied that
\begin{equation}
\label{sum-beta-A-B}
\sum _{i \in A \backslash B } \beta _i < \varepsilon  \sep \Rightarrow \sep   \sum
_{i \in B } \beta _i > 1 - \varepsilon > \frac{1}{2}.
\end{equation}
We will  check now that the element $v_{0}  =\frac{1}{ \sum _{i \in B} \beta _i }
\sum _{i \in B } \beta _i v_i$ satisfies the requirements of the claim. Since  $B
\subset \{ 1,2,3,4\}$,  clearly $v_0\in  \co \{v_1, v_2, v_3, v_4\} \subset E_1$.

Since $B \subset A$ the element $v_0$ also  satisfies
\begin{eqnarray*}
    \Vert v_{0} - u_0 \Vert &\le&  \biggl\Vert   \frac{1}{\sum _{ i \in B}  \beta _i } \sum _{ i \in B}  \beta _i v_i   - \sum _{ i \in B }  \beta _i v_i  \biggr\Vert + \Bigl \Vert \sum _{ i \in A \backslash B}  \beta _i v_i \Bigr\Vert \\
    &\le& \frac{1}{\sum _{ i \in B }  \beta _i} -1 +  \sum _{ i \in A \backslash B}  \beta _i \\
    &=& \frac{1}{ \sum_{i \in B} \beta _i }   \sum _{ i \in A \backslash B}  \beta _i +  \sum _{ i \in A \backslash B}  \beta _i  \\
    &=&  \sum _{ i \in A \backslash B}  \beta _i  \biggl( 1 + \frac{1}{  \sum _{ i \in B}  \beta _i}\biggr) \\
    &<& \varepsilon  \biggl( 1 + \frac{1}{ 1- \varepsilon } \biggr) < 3 \varepsilon  \sem \text{(by \eqref{sum-beta-A-B})}.
% &  < 3 \varepsilon
\end{eqnarray*}
Let us notice that $B \subset \{ 1,2,3,4\} \cap A$. If we take  $\gamma_i= \frac{ \beta _i}{
\sum _{i \in B} \beta_i}$ for every $i \in B$ then  $v_0 =  \sum _{i \in B} \gamma _i v_i$.  As a consequence, in case that 
$\gamma _i > 0 $ for some $i \in B$ we obtain that $\beta _i > 0$. So the element $v_0$ satisfies all the required conditions.
\end{proof}

%{\bf The previous results will be used in order to prove that BPBp for the pair %$(\ell_\infty^4,Y)$ implies that $Y$ has the AHSp-$\ell_\infty^4$. TO BE DELETED}

 Next result is the version of \cite[Theorem 2.9]{ABGKM2} for $ \ell _\infty ^4$, where the analogous result was obtained for $\ell_\infty ^3$.
In our case, the fact that  the domain has dimension $4$, and so the  norm of an element $T \in L(\ell_\infty ^4, Y)$ is the maximum
 of the norm of the evaluation of $T$ at eight extreme points of $B_{\ell_\infty^4}$ makes the proof more  complicated comparing to the case that the domain has dimension $3$.

\vskip10mm

\begin{theorem}
\label{teo-char}
    For every  Banach space $Y,$ the pair $(\ell_\infty^4,Y)$ has the  Bishop-Phelps-Bollobás property if and only if $Y$ has the approximate hyperplane sum property for $\ell_\infty^4$.
    
     Moreover,  if $(\ell_\infty ^4, Y) $ satisfies Definition \ref{def-BPBp} with the function $\eta $, then $Y$ has the AHSp-$\ell_\infty ^4 $ with 
     \linebreak[4]
     $\gamma (\varepsilon )= \eta \bigl( \frac{\varepsilon }{48}\bigr) $. In case that  $Y$ has the AHSp-$\ell_\infty ^4 $ for the function  $\gamma $ (see Definition \ref{def-AHSPL4}), the pair
    	    	 $(\ell_\infty ^4, Y) $ satisfies  BPBp with  the function $\eta (\varepsilon)= \gamma^2 \bigl( \frac{\varepsilon }{3} \bigr) $. 
  \end{theorem}
\begin{proof}
Assume that the pair  $(\ell_\infty ^4, Y)$ satisfies the BPBp with the function $ \eta $. We will prove that $Y$ satisfies condition 3) in Proposition \ref{pro-char} with the function $\nu(\varepsilon)= \eta \bigl( \frac{ \varepsilon}{12}\bigr) $.  As a consequence $Y$ has the AHSp-$\ell_\infty ^4$ with the function $\gamma (\varepsilon)= \eta 
 \bigl( \frac{ \varepsilon}{48}\bigr) $.

Let us fix $0 < \varepsilon < 1$. 
Assume that  $(y_i )_{i \le 4} \in M_Y^4$  and  that
$\sum_{i=1}^{4} \alpha_i y_i $ is a convex combination  satisfying that 
$$
\Big\Vert  \sum_{i=1}^{4} \alpha_iy_i \Big\Vert  >  1 - \nu(\varepsilon).
$$
Let $T$ be the element in $B_{L(\ell_\infty ^4,Y)}$   that $(y_i)_{i \le 4}$ represents  in view of  Proposition  \ref{prop-ident}. The
element
\linebreak[4]
 $x_0= \sum_{i=1}^{4} \alpha_i v_i$ satisfies that $x_0 \in
S_{\ell_\infty^4}$ and by assumption we know that
$$
\bigl\Vert T(x_0) \bigr\Vert = \Bigl \Vert \sum_{i=1}^{4} \alpha_i y_i   \Bigr
\Vert >    1 - \nu(\varepsilon) = 1 - \eta\Big(\frac{\varepsilon}{12}\Big)> 1 - \frac{\varepsilon}{12} >   0.
$$
Since the pair  $(\ell_\infty ^4, Y)$ has the BPBp, there are  $u_0\in
S_{\ell_\infty^4}$  and   $S\in S_{L(\ell_\infty^4,Y)}$   satisfying the
following conditions
\begin{equation}
\label{Su0}
\Vert Su_0\Vert =1,  \sem \Vert u_0 - x_0\Vert <\frac{\varepsilon}{12}  \sem
\text{and}  \sem \Big \Vert S-\frac{T}{\Vert T\Vert } \Big \Vert
<\frac{\varepsilon}{12} .
\end{equation}

Notice that  $ \vert u_0(1) - 1\vert  =\vert u_0(1) - x_0(1)\vert \leq \Vert  u_0 -
x_0\Vert  < 1$ and $\Vert u_0\Vert =1$, so  $ 0 < u_0(1) \le 1$. Since we clearly
have that the element $u_0$ can be written as a convex combination as follows 
$$
u_0= \frac{1+u_0(1)}{2}(1,u_0(2),u_0(3),u_0(4)) +
\frac{1-u_0(1)}{2}(-1,u_0(2),u_0(3),u_0(4)),
$$
$ 1 + u_0(1) > 0$ and $S$ attains its norm at $u_0$, then  $S$ also attains its norm at
$(1,u_0(2),u_0(3),u_0(4))$.  The previous element  belongs to $S_ {\ell_\infty ^4}
$ and satisfies that
$$
\Vert (1,u_0(2),u_0(3),u_0(4)) - x_0\Vert  \leq \Vert u_0- x_0\Vert  <
\frac{\varepsilon}{12}.
$$
As a consequence, by changing $u_0$ by $(1, u_0(2), u_0(3), u_0(4))$, if needed, we can assume that $u_0 \in E_1$.

By Lemma \ref{le-close-face} there is $v_0 \in E_1$ such that $\Vert v_0 - u_0
\Vert < \frac{\varepsilon}{4}$ and such that there is  a set 
\linebreak[4]
$A \subset \{ i \in \N
: i \le 8\}$ satisfying that   $u_0$ and  $v_0$ can be written as convex
combinations as follows
$$
u_0= \sum _{i \in A} \beta _i v_i , \sep v_0= \sum _{i \in A, i \le 4} \gamma _i
v_i
$$
and also
\begin{equation}
\label{gamma-beta}
  \gamma _i > 0 \sep \text{\rm for some} \sep i \ \Rightarrow \ \beta _i > 0.
\end{equation}
By \eqref{Su0} $S$ attains its norm at $u_0$, hence  by Hahn-Banach Theorem there
is an element $y^* \in S_{Y^*}$ such that $y^* (S(u_0))=1$. Since
$$
1=y^* (S(u_0)) = \sum _{i\in A} \beta _i y^* S(v_i),
$$
$\beta _i \ge 0 $ for each $i \in A$ and $\sum _{ i\in A} \beta _i=1$, we have that
$$
 \beta _i \in A, \ \beta _i > 0 \ \Rightarrow \ y^* \bigl( S(v_i) \bigr)=1.
 $$
As a consequence, by  \eqref{gamma-beta} we obtain that
\begin{equation}
\label{y*Svi}
i \in A, \  i \le 4,\  \gamma_i \ne 0 \  \Rightarrow \  y^* \bigl( S(v_i)\bigr) =
1.
\end{equation}
The element $v_{0 }$   satisfies  that
\begin{align}
\nonumber
\Vert  v_0 - x_0 \Vert  & \leq \Vert v_0 - u_0 \Vert  + \Vert  u_0 - x_0\Vert\\
%\nonumber
\label{v0-x0}
& <  \frac{\varepsilon}{4} + \frac{\varepsilon}{12}  \seg \text{(by \eqref{Su0})}\\
\nonumber & = \frac{ \varepsilon}{3}.
\end{align}
By Lemma \ref{le-basis} the functionals $\{ v_{i}^* : 1 \le i  \le 4\} $ belong to
$B_{(\ell_\infty ^4)^* }$ and are the biorthogonal functionals of the basis
$\mathcal{B}$.  In view of \eqref{v0-x0} we get that
\begin{equation}
\label{alphai-gammai}
\vert \alpha _i - \gamma _i \vert = \vert  v_{i}^* (x_0- v_{0}) \vert \le
\Vert x_0 - v_{0} \Vert < \frac{\varepsilon}{3} , \sem \forall i \le 4.
\end{equation}

We write $C:= \{ i\leq 4 : i \in A,  \gamma_i \neq 0\}$.  If $C=\{1,2,3,4\}$ then $\sum_{i
\in C} \alpha _i=1$. Otherwise, from \eqref{alphai-gammai} we obtain that
$$
\sum_{i\in C}\alpha_i= 1- \sum_{i\le 4, i \not\in C}\alpha_i > 1
-\frac{\varepsilon}{3} ( 4 -  \vert C \vert ) \geq 1 - \varepsilon.
$$
Finally we check that $(z_i)_{i \le 4}=(S(v_i))_{i \le 4}$ is the desired  element in $M_Y^4$.  Clearly $(z_i)_{i \le 4} \in M_Y^4$  since $S\in S_{ L(\ell_\infty^4, Y)}$ (see Proposition  \ref{prop-ident}).  By
\eqref{y*Svi} we have that
$$
\Big\Vert  \sum_{i\in C}z_i   \Big\Vert  \ge y^*  \biggl( \sum_{i\in C}  S(v_i)
\biggr) = \vert C \vert.
$$
It remains to show only that $\Vert z_i - y_i \Vert < \varepsilon $ for each $i \le
4$. Indeed for each $1 \le i \le 4$ we have that
\begin{align*}
 \Vert  z_i - y_i\Vert  & = \Vert  S(v_i)- T(v_i) \Vert   \\
 & \leq \Big\Vert  S(v_i) - \frac{T(v_i)}{\Vert T\Vert }  \Big\Vert  + \Big\Vert
\frac{T(v_i)}{\Vert T\Vert } - T(v_i)  \Big\Vert        \\
 & \leq \Big\Vert  S -  \frac{T}{\Vert T\Vert }  \Big\Vert  + 1 - \Vert T\Vert
  \\
& < \frac{\varepsilon}{6}  < \varepsilon  \sem \text{(by \eqref{Su0})}.
\end{align*}
 We proved that   $Y$ satisfies condition 3) of Proposition \ref{pro-char} with $\nu (\varepsilon ) = \eta\bigl(\frac{\varepsilon}{12}\bigr)$.

Assume that  $Y$ satisfies the  AHSp-$\ell_\infty ^4$ with the function $\gamma $.   By Proposition  \ref{pro-char}, $Y$ also  satisfies condition 3) in that result with the function $\nu (\varepsilon) = \gamma ^2 (\varepsilon )$.  We will show that that the pair  $( \ell _\infty ^4, Y)$ has the BPBp with the function $\eta (\varepsilon) = \nu(\frac{\varepsilon}{3})= \gamma  ^2 (\frac{\varepsilon}{3})$.

 Let be  $0<\varepsilon <1$ and
assume that   $T\in S_{L(\ell_\infty^4,Y)}$ and   $x_0 \in
S_{\ell_\infty^4}$  are such that
$$
\|Tx_0\| >1-\eta (\varepsilon) = 1- \nu \Bigl (\frac{\varepsilon}{3} \Bigr).
$$
By using an appropriate isometry, if needed, in view of Lemma \ref{le-faces},  we can assume that
\linebreak[4]
 $x_0 \in
\text{co}\{v_i : 1 \le i \le 4\}$.  Let us write $x_0$ as
a convex combination $x_0=\sum_{i=1}^{4}\alpha_i v_i$. In view of
\linebreak[4]
 Proposition  \ref{prop-ident} the set $(y_i)_{i \le 4}= (T(v_i))_{i \le 4} \in M_Y^4$.  So we have that 
$$
\biggr \Vert \sum_{i=1}^{4}\alpha_i y_i \biggl\Vert  = \Vert T (x_0) \Vert  >1-\nu
\Big(\dfrac{\varepsilon}{3}\Big).
$$
By assumption $Y$ satisfies condition  3) in Proposition \ref{pro-char}, so there is a (nonempty) set  $A\subset
\{1,2,3,4 \}$, and $(z_i)_{i \le 4} \in M_Y^4$  such that
%y $z^{*}\in S_{Y^{*}}$ tal que
\begin{equation}
\label{sum-A-alphai}
\sum_{i\in A}\alpha_i>1-\frac{\varepsilon}{3} > 0, \seg
\|z_i-y_i\|<\frac{\varepsilon}{3} \sep \text{for all}\  i \le 4
\end{equation}
and
\begin{equation}
\label{A-zi-sum}
\Bigl\Vert \sum _{ i \in A} z_i \Bigr \Vert = \vert A \vert .
\end{equation}

%\begin{enumerate}
%   \item[i)] $\sum_{i\in A}\alpha_i>1-\frac{\varepsilon}{3},$
%   \item[ii)]  $\|z_i-y_i\|<\frac{\varepsilon}{3} \ \ \ \text{para cada } i \le 4,$
%   \item [iii)]  $z^\ast(z_i)=1 \ \ \  \text{para cada } i\in A.$
%\end{enumerate}
Let $S$ be the unique linear operator from $\ell_{\infty}^4$ to $Y$  such that
$S(v_i)=z_i$ for every  $1\leq i \leq 4$.  Since   $(z_i)_{i \le 4} \in M_Y^4$,   $S \in B_{L(\ell_\infty^4,Y)}$ by Proposition  \ref{prop-ident}. The
element $u_0$ given by $u_0=\sum_{i\in A}\frac{\alpha_i}{\sum_{i\in A}\alpha_i}v_i$
belongs to $S_{\ell_{\infty}^4}$.  By  equation \eqref{A-zi-sum} the operator $S$
attains its norm at $u_0$ since
$$
%\begin{equation}
%\nonumber
1  =  \frac{ \Vert \sum _{i \in A} \alpha _i z_i \Vert}{  \sum _{i \in A} \alpha _i
}   = \Vert S (u_0) \Vert \le \Vert S \Vert \le 1.
$$
So $S \in S_{L(\ell_\infty^4,Y)}$. From Proposition  \ref{prop-ident} and
\eqref{sum-A-alphai} we obtain that $\Vert S - T \Vert < \varepsilon $.  We write
\linebreak[4]
$C=\{  i \in \N: i\le 4, i \notin A\}$. Finally we
obtain that
\begin{align*}
 \Vert  x_0 - u_0\Vert  & =  \Bigl\Vert  \sum_{i=1}^{4}\alpha_i v_{i} - \sum_{i\in
A}\dfrac{\alpha_i}{\sum_{i\in A}\alpha_i}v_i
 \Big\Vert \\
& =  \Big \| \Big(1-\frac{1}{\sum_{i\in A}\alpha_i}\Big) \sum_{i\in A}\alpha_iv_i + \sum_{ i \in C  }\alpha_i v_{i}\Big \| \\
& \leq \Big| 1-\frac{1}{\sum_{i\in A}\alpha_i}\Big| \sum_{i\in A}\alpha_i\|v_i\| + \sum_{i \in C} \alpha_i \|v_{i}\| \\
& \leq \frac{\sum_{i \in C}\alpha_i}{\sum_{i\in A}\alpha_i}\sum_{i\in A}\alpha_i+ \sum_{i \in C}\alpha_i \\
&= 2 \sum_{i \in C }\alpha_i<  \frac{ 2\varepsilon}{3} < \varepsilon  \seg \text{(by \eqref{sum-A-alphai})}.
\end{align*}
 We proved that the pair $(\ell_\infty^4, Y)$  has the BPBp with $\eta (\varepsilon)= \nu \bigl( \frac{\varepsilon}{3} \bigr)= \gamma ^2 \bigl(\frac{\varepsilon}{3} \bigr) $.
\end{proof}

%INCLUIR UN RESSULTADO QUE DIGA QUE SI $(c_0,Y)$ tiene BPBp, entonces   %$(\ell_\infty^n,Y)$ tienen BPBp uniformemente.

\section{Examples of spaces with the approximate  hyperplane  sum property for       $\ell_\infty^4$}

The goal of this section  is to provide classes of  Banach spaces with the approximate  hyperplane sum property for $\ell_\infty ^4$.

As we already mentioned in the introduction  the pair $(X,Y)$ has  BPBp whenever $X$ and $Y$ are finite-dimensional normed spaces \cite[Proposition 2.4]{AAGM1}.  By applying this result to $X= \ell_\infty ^4$, and  in view of Theorem \ref{teo-char} we obtain that finite-dimensional spaces have the AHSp-$\ell_\infty ^4$. We will also  provide a simple direct proof of this fact.

\begin{proposition}
			\label{pro-finite}
				Every finite-dimensional normed space has the  AHSp-$\ell_\infty^4$.
\end{proposition}

\begin{proof}
We will argue by contradiction. 
	Let $Y$ be a finite-dimensional space and assume that $Y$ does not have   the   AHSp-$\ell_\infty^4$.  So there is $\varepsilon _0 > 0$   for which Definition \ref{def-AHSPL4} is not satisfied. Hence   there is  a sequence $\{\gamma_n \}$ of positive  real numbers satisfying $ \{\gamma_n \} \to 0 $ and also   for each natural number $n$,  there  are an element  $(y_i^n)_{i \le 4} \in  M_{Y}^4$  and a nonempty  set $A_n \subset  \{ 1,2,3,4\}$  satisfying 
	\begin{equation}
	\label{sum-A}
	\Bigl\Vert \sum_{i \in A_n } y^n_i \Bigr \Vert > |A_n| \bigl( 1 - \gamma _n \bigr)
	\end{equation}
	and also
	\begin{equation}
	\label{dist-z-y}
	(z_i)_{i \le 4} \in M_{Y}^4, \sep \Bigl\Vert \sum_{i \in A_n } z_i \Bigr\Vert = \vert A_n \vert  \sep \Rightarrow \sep
	\max \{ \Vert y_i^n-z_i \Vert  : i\leq 4 \}\ge  \varepsilon _0,
	\end{equation}
%	where $\tri \ \tri$  denotes the norm on $Y^4$ given by
	$$
%	\tri (y_i)_{i \le 4} \tri = \max \{ \Vert y_i \Vert : i \le 4\}.
	$$
%(\textbf{IN (4.2) , INSTEAD OF TRIPLE NORM WE USE MAX %NORM} )	.
Since $Y$ is finite dimensional, $M_{Y}^4$ is a compact set of $Y^4$. By taking into account that  $\{i \in \N: i \le 4\}$ is finite,   and  passing to a subsequence, we can assume that there is a set $A \subset \{i \in \N: i \le 4\}$ such that $A_n=A$, for every $n$,  and also that  for each  $i \le 4$ the sequence $\{y_i^n\}_n$ converges to $y_i$, so $(y_i)_{i \le 4} \in M_{Y}^4$. From condition \eqref{sum-A} it follows that $\bigl \Vert \sum _{i \in A } y_i  \bigr \Vert = \vert A \vert $.   In view of  condition \eqref{dist-z-y} this is a contradiction.
\end{proof}

Recall that a Banach space $Y$ is {\it uniformly convex } if for every $ \varepsilon
> 0$ there is $ 0 < \delta < 1$ such that
$$
u,v \in B_Y, \ \ \frac{\Vert u + v \Vert }{2}
> 1 - \delta \ \ \Rightarrow \ \  \Vert u - v \Vert < \varepsilon \ .
$$
In such a case, {\it the modulus of convexity} of $Y$ is the function defined by
$$
\delta_Y (\varepsilon):= \inf \Bigl\{ 1 - \frac{\Vert u + v
	\Vert}{2} : u,v \in B_Y, \Vert u - v \Vert \ge \varepsilon
\Bigr\} \ .
$$

As a consequence of \cite[Theorem 2.5]{KiI} the pair $(\ell_\infty ^4, Y)$ has the Bishop-Phelps-Bollob{\'a}s property  for operators in  case that $Y$ is  uniformly convex. By  Theorem 
\ref{teo-char}  uniformly convex spaces have AHSp-$\ell_\infty ^4$.  However, we provide a direct proof of that  fact.

\begin{lemma}
	\label{le-3-points}
	Assume that $ (y_i)_{ i \le 4} \in M_Y ^4$, $y^* \in S_{Y^*}$ and $\delta > 0$ satisfies that
	$$
	i< k \le 4, \sep  y^* (y_i), y^* (y_k ) > 1 - \delta \ \Rightarrow \ y^* (y_j) > 1 - 2 \delta \sep \text{\rm for each } \sep i < j < k.
	$$
\end{lemma}
\begin{proof}
	Since 	 $( y_i)_{ i \le 4} \in M_{ Y }^4$, then for every $i < j < k$ we have that $\Vert y_i-y_j + y_k \Vert \le 1$.  If we assume that  $ y^*(y_i), y^* (y_k) > 1 - \delta$ then  we have that
	$$
	2- 2\delta - y^{*}(y_j)<y^{*}(y_i-y_j+y_k)\leq \Vert y_i-y_j+y_k\Vert \leq 1.
	$$
	As a consequence,
	$$
	y^{*}(y_j) > 1 - 2  \delta.
	$$	
\end{proof}

\begin{proposition}
	\label{pro-uc}
	Every uniformly convex   Banach space has the approximate  hyperplane  sum property for
	$\ell_\infty^4$.
 Moreover, if $\delta_Y$ is the modulus of convexity of $Y$, then $Y$
		has the AHSp-$\ell_\infty^4$ with the function  $\gamma(\varepsilon)=  \min \bigl\{ \frac{\delta_Y (\varepsilon)}{2},\frac{\varepsilon}{6} \bigl\}$.
\end{proposition}
\begin{proof}
	Assume that $Y$ is a   uniformly  convex  Banach space with  modulus of convexity $\delta_Y$.  Given
	\linebreak[4]
	 $0 <
	\varepsilon < 1$, we define $\gamma(\varepsilon)=  \min \bigl\{ \frac{\delta_Y (\varepsilon)}{2},\frac{\varepsilon}{6} \bigl\}$. Assume that $y^\ast \in S_{Y^\ast }$, $(y_i)_{i \le 4} \in M_Y^4$  and 
%	\linebreak[4]
	 $\varnothing \ne A \subset \{ 1,2,3,4 \}$  is a set such
	that
	$$
	y^* (y_i)
	> 1 - \gamma(\varepsilon) > 0 , \sep \forall  i\in A.
	$$
	
By Lemma \ref{le-t-sets} we can assume that $\min A  =1$. In view of Lemma \ref{le-3-points}, the set 
\linebreak[4]
$C=\{ i\in \N: 1\le i \le \max A\}$ satisfies that	
$$
	y^* (y_i)
	> 1 -2 \gamma(\varepsilon) > 0 , \sep \forall  i\in C.
	$$
	
	Hence for every  $i,j \in C$
	we have that
	$$
	1-  \delta_Y (\varepsilon ) \le 1 - 2\gamma(\varepsilon) < y^* \Biggl( \frac{y_i +   y_j}{2} \Biggr)\le y^* \Biggl( \frac{  \frac{ y_i}{ \| y_i \| } + y_j  }{2} \Biggr)  \le \Biggl \Vert
	\frac{ \frac{ y_i}{ \| y_i \| } + y_j }{2}
	\Biggr\Vert.
	$$
	By the definition of the
	modulus of convexity it follows that 
	\begin{equation}
	\label{zi-yi-U}
	\Bigr\Vert   \frac{y_i}{\| y_i \|} - y_j \Bigr\Vert < \varepsilon, \seg \forall i,j \in C.
	\end{equation}
	We will show that there exists $(z_i)_{i \le 4} \in M_Y^4$
	satisfying that
	\begin{equation}
	\label{zi-yi-A}
	\Vert  z_i-y_i \Vert < \varepsilon, \forall i \le 4 \sem \text{   and } \sem \Big \Vert \sum_{i \in C } z_i\Big  \Vert = \vert C\vert .
	\end{equation}

	By Proposition \ref{pro-AHSP-1-element}  it suffices to show \eqref{zi-yi-A}  in case that  there is $2 \le k \le 4$ such that
	% the set 
	$C$ coincides with  the set $C_k=\{ i \in \N:  i \le k\}$.

	So let us fix $2 \le k \le 4$  and define  $ z_i= \frac{y_{1}}{\Vert y_{1} \Vert } $ for every $i \in C_k$. From \eqref{zi-yi-U} it follows that
	\begin{equation}
	\label{zi-yi-U2}
	\Vert z_i- y_i \| < \varepsilon, \seg \forall i \in C_k
	\end{equation}
	and it is trivially satisfied that 
	\begin{equation}
	\label{sum-Ak}
	%$$
	\Big \Vert \sum_{i \in C_k } z_i\Big  \Vert = \vert C_k\vert  .
	\end{equation}

	In case that $k=2$, we put  $a= \frac{\gamma(\varepsilon)}{1+\gamma(\varepsilon)}, b= \frac{\gamma(\varepsilon)}{2(1+\gamma(\varepsilon))}$ and
	$$
	z_3 =
	(1-a) y_3 + b z_1 , \sem z_4 =  (1-a) y_4 - b z_1.
	$$
	For each $i\in \{3,4\}$, it follows that
	$$
	\Vert z_i \Vert \le 1-a +b = 1- \frac{\gamma(\varepsilon)}{1+\gamma(\varepsilon)} +  \frac{\gamma(\varepsilon)}{2(1+\gamma(\varepsilon)) } \le 1
	$$
	and
	$$
	\Vert z_i - y_i \Vert \le a+b = \frac{\gamma(\varepsilon)}{1+\gamma(\varepsilon)} + \frac{\gamma(\varepsilon)}{2(1+\gamma(\varepsilon))} = \frac{3\gamma(\varepsilon)}{2(1+\gamma(\varepsilon)) } <
	\frac{3\gamma(\varepsilon)}{2} < \varepsilon.
	$$
	By  the last chain of inequalities and \eqref{zi-yi-U2} it is satisfied  that 
	$$
	\Vert z_i- y_i \| < \varepsilon, \seg \forall i \le 4.
	$$
	Since  for each $i\in \{3,4\}$ we have
	\begin{equation}
	\label{zi-yi-U3}
	\Vert z_1-z_2 + z_i \Vert = \Vert z_i \Vert  \le 1 
	\end{equation}
	and   for each  $j\in\{1,2\}$  it is clear that 
	\begin{align}
	\label{zi-yi-U4}
	\nonumber 
	\Vert z_j - z_3 + z_4 \Vert & = \Vert z_1 - z_3 + z_4 \Vert   \\
	\nonumber 
	& = \Vert (1-2b)  z_1 + (1-a) (-y_3 + y_4 )   \Vert  \\
	\nonumber  
	& = (1-a)
	\Vert   z_1 - y_3 + y_4 \Vert \\
	& \leq (1-a)  ( \Vert y_1 -y_3 + y_4 \Vert  + \Vert z_1 - y _1 \Vert ) \\
	\nonumber
	& \le    (1-a) \Bigl( 1 + \Bigl \Vert \frac{y_1} {\Vert y_1 \Vert } - y_1 \Bigr \Vert \Bigr) \\
	\nonumber 
	&  =(1-a)(2-\|y_1\|)<(1-a)(1+\gamma(\varepsilon))  = 1.
	\end{align}
	In view of \eqref{zi-yi-U3} and \eqref{zi-yi-U4}
	we obtain that $ (z_i)_{ i \le 4} \in M_Y ^4$.
	
	If $k=3$, we take $z_4=y_4$, in this case it is immediate to check that $ (z_i)_{ i \le 4} \in M_Y ^4$, also by the definition of $z_4$ and \eqref{zi-yi-U2} it follows that
	$$\Vert z_i- y_i \| < \varepsilon, \seg \forall i \le 4.
	$$
	Finally, if $k=4$, we have by definition that $ z_i=
	\frac{y_{1}}{\Vert y_{1} \Vert } $ for every $i \in \{1,2,3,4\}$, so it is  trivially satisfied  that  $ (z_i)_{ i \le 4} \in M_Y ^4$ and  in view of  \eqref{zi-yi-U2}  we also have that  $\Vert z_i- y_i \| < \varepsilon$, for each $i \in \{1,2,3,4\}$.

	We showed  that for every $2 \le k \le 4$ there exists  $ (z_i)_{ i \le 4} \in M_Y ^4$ satisfying  \eqref{zi-yi-A} for $C= C_k$,  so the proof is complete since $A\subset C$.
\end{proof}

As a consequence of \cite[Theorem 2.4]{AGKM}, \cite[Theorem 2.1]{ACKLM} and Theorem \ref{teo-char},  there is also  a  nontrivial class of spaces  containing  uniformly convex spaces and satisfying AHSp-$\ell_\infty ^4 $.

%\vskip10mm

The next statement can be shown  by using the same argument of the analogous  result  for AHSp-$\ell_\infty^3$ (see  \cite[Proposition 2.4]{ABGKM2}). For this reason we do not include the proof.

\begin{proposition}
		\label{pro-C0-Y}
		Let $L$ be a nonempty locally compact Hausdorff topological space and $Y$ a Banach space. Then
$\mathcal{C}_0(L,Y)$ has the AHSp-$\ell_\infty^4$ if, and only if, $Y$ has the
AHSp-$\ell_\infty^4$.
\end{proposition}

It is clear that  $\R$  has the  AHSp-$\ell_\infty^4$, so  
 by the previous result $\mathcal{C}_0(L)$ has the same property for any locally compact Hausdorff space $L$.

\vskip2mm

Our aim now is to prove that the space $\ell_1$ has the   AHSp-$\ell_\infty^4$. 
In what follows we will denote by  $u^* $ the element in $\ell_{1}^* $ given by
$$
u^* (x)=\sum _{n=1}^\infty x(n) \seg (x \in \ell_1).
$$

 The next simple result will be useful for this purpose.

\begin{lemma}
\label{lema-auxiliar}
Let be  $s,t \in \R^+$, $x,y,z \in \ell_{1}$. Assume that
$$
1-s \le u^*(x-y+z)  \sem \text{and}\sem \Vert x-y+z \Vert \le
1+t.
$$
Then there is  $w\in \ell_{1}$  such that
$$
w\ge z ,  \seg
\Vert w-z\Vert \le s+t  \seg \text{and}   \seg  x-y+w \ge 0.
$$
\end{lemma}
\begin{proof}
Define the sets  given by
$$
P= \{k\in \N: y(k)\le(x+z)(k)  \} \seg \text{and} \seg  N=\N  \backslash  P.
$$
Since $u^* \ge 0$,  we have that
\begin{eqnarray*}
1-s & \le & u^*(x-y+z) \le u^* \bigl( (x-y+z)\chi_P \bigr)    \\
&\le&  \Vert (x-y+z)\chi_P  \Vert \le \Vert x-y+z \Vert \\
&\le &  1 + t .
\end{eqnarray*}
Hence $\Vert (x-y+z)\chi_P\Vert \ge 1-s$, so
\begin{equation}
\label{xyz-N}
\Vert(x-y+z)\chi_N\Vert= \Vert x-y+z \Vert - \Vert (x-y+z)\chi_P\Vert \le 1+t - (1-s) = s+t.
\end{equation}
Let  $w\in \ell_1$  be the element given by
$$
w(k) =
\begin{cases}
z(k)            & \mbox{if} \sep k \in P   \\
(y-x)(k)  & \mbox{if} \sep   k \in N.
\end{cases}
$$
It is clear that
$
w\chi_P=z\chi_P $ and $(y-x)\chi_N \ge  z\chi_N$, so $w \ge z$.
Since
$$
(x-y+w)\chi_P=(x-y+z)\chi_P\ge 0 \seg \text{and} \seg  (x-y+w)\chi_N=0,
$$
it is satisfied that  $x-y+w\ge 0$.

In view of \eqref{xyz-N} we  also have that
\begin{align*}
\Vert w-z \Vert& = \Vert(w-z)\chi_N\Vert=\Vert(y-x-z)\chi_N\Vert \\
&  =\Vert(x-y+z)\chi_N\Vert \le s+t.
\end{align*}
\end{proof}

  Notice that in Lemma \ref{lema-auxiliar} the element $x$ satisfies the same assumptions that $z$. So under the same conditions we also obtain an element $v\ge x$ such that  $\Vert v-x\Vert \le s+t$  and $v-y+z \ge 0$.

\begin{proposition}
\label{prop-ell1}
There is a function $\rho : ]0,1[ \llll \R^+$ such that  the  space $\ell_{1}^n$  satisfies condition   2) of Proposition \ref{pro-char}  for such function, for each natural number  $n$.  Also the  space $\ell_1$ satisfies the previous statement. As a consequence, there is a function $\gamma :  ]0,1[ \llll \R^+$  such that the spaces $\ell_{1}^n$ and $\ell_1$   have  the approximate hyperplane sum property for $\ell_\infty^4$ with such function  (see Definition \ref{def-AHSPL4}).
\end{proposition}
\begin{proof}
 We prove the statement for $\ell_1$.  To this purpose we denote by $\{ e_n\}$  the usual Schauder basis of $\ell_1$. 
It suffices to  show  that $\ell_1$ satisfies  condition 2)  in  Proposition \ref{pro-char} for $\rho(\varepsilon)= \frac{\varepsilon
}{226}$ and  $E=
\ext (B_{(\ell_{1})^*  })$, which is clearly  a $1$-norming set for $\ell_1$.  In such case we deduce that  $\ell_1$ satisfies the AHSp-$\ell_\infty ^4$  with the function $\gamma $ given by  $\gamma (\varepsilon)= \rho ^2  (\frac{\varepsilon}{4})$ by  Proposition \ref{pro-char}.

 From now on $0 < \varepsilon < 1$ will be fixed and we simply write  $\rho$ instead of $\rho(\varepsilon)$.
Assume that 
\linebreak[4]
 $y^* \in E= \ext (B_{(\ell_{1})^*  })$,    $(a_i)_{i \le 4} \in M_{\ell_1}^4$   and  $\varnothing \ne C \subset \{ 1,2,3,4 \}$  is a  nonempty set such that
$$
y^* (a_i) > 1 -\rho  , \sep \forall  i\in C.
$$
Firstly by Lemma \ref{le-t-sets} we can clearly assume that $\min C =1$. 

In the case that $C$ contains only one element,   by Proposition \ref{pro-AHSP-1-element}  condition 2) in Proposition \ref{pro-char} is satisfied for such set $C$.
Otherwise $C$ contains at least two elements.  

Since $\ext (B_{(\ell_{1})^*  }) = \{ z^* \in \ell_{1}^* : \vert  z^* (e_n) \vert =1, \forall n \in \N\}$, by using an appropriate surjective linear isometry onto $\ell_1$, we can clearly assume that $y^*= u^*$. It suffices to show that  there exists   $(z_i)_{ i \le  4} \in M_{ \ell_1}^4$ satisfying that
$$
\Vert  z_i - a_i \Vert < \varepsilon, \forall i \le 4 \sem \text{   and also} \sem
u^* (z_i) =1  \sem \text{and} \sem z_i \ge 0,  \sep \forall  i\in C.
$$

We define the set $C^\prime $ given by
$$
C':= \{ i\in \N: 1\leq i\leq \max C, \ u^* (a_i) > 1 - 2\rho \}.
$$
In case that  $i,k \in C$ and  $j \in \N$   satisfies  $ i  <  j < k$, by Lemma \ref{le-3-points} we know that $j \in C^\prime$. Since $C \subset  C^\prime$,  the previous remark shows that
$C^\prime $  has consecutive elements.
%is the union of $A$ and the integers in the interval  $[i_0,  %\max A]$,
In fact
$$
C^\prime = \{ i \in \N:  1 =\min C \le i \le \max C\}.
$$

We define
$$
P_i=\{ k \in \N : a_i(k) \ge 0\},  \seg   N_i=\N \backslash P_i,  \seg b_i = a_i \chi_{P_i},  \seg (i \in  C')
$$
Since $u^* \ge 0$, by assumption we have that
\begin{equation}
\label{ai-big}
1-2\rho < u^*(a_i) \le  u^*(b_i) \le \Vert b_i \Vert \le \Vert a_i \Vert \le 1 \sep \text{and} \sep b_i \ge 0,  \sem  \forall i\in C'.
\end{equation}
Hence
\begin{equation}
\label{ai-ui}
\Vert b_i-a_i \Vert = \Vert a_i  \chi_{N_i} \Vert  \le 1 - \Vert b_i \Vert < 2\rho, \quad \forall i\in C'.
\end{equation}

\noindent
Now we will distinguish several cases,  depending on the number of elements of $C^\prime$. Since we   assume that $C$ contains more than one element and $C \subset  C^\prime $, then $C^\prime $ also contains  at least two elements.  So  $C'=\{1,2\}$, $C'=\{1,2,3\}$ or $C'=\{1,2,3,4\}$.

\noindent  $\bullet$ \underline{Case 1:} Assume  that   $C'=\{ 1,2\}$.
Define the elements in $\ell_1$  given by
$$
x_i =
\begin{cases}
b_i + (1-\|b_i\| )e_1           & \mbox{if} \sep  i \in \{1,2 \}   \\
a_i  & \mbox{if} \sep  i \in \{ 3,4 \} .
\end{cases}
$$
So we have that
$$
x_i\ge 0, \sem \text{and} \sem \Vert x_i \Vert =1, \sep  i=1,2.
$$
From \eqref{ai-big} and   \eqref{ai-ui} it follows that
\begin{align}
\label{bi2}
%\label{xi-ai}
 \sem \Vert x_{i}  \Vert \le  1 \sem \text{and }\sem   \Vert x_{i}-a_{i}
\Vert < 4\rho, \seg \forall  i \le 4.
\end{align}
In this case we will change twice the previous vectors   in order to get the desired properties. Firstly  we define the set $\{ y_i: 1 \le i \le 4 \} $ by
$$
y_i =
\begin{cases}
\frac{x_i}{1+ 8\rho} + (1-\frac{1}{1+8\rho})e_1         & \mbox{if} \sep i \in \{ 1,2 \}   \\
\frac{x_i}{1+8\rho} & \mbox{if} \sep  i \in \{3,4 \} .
\end{cases}
$$
It is clearly satisfied that
\begin{equation}
\label{ci-second-case}
y_{i} \geq 0 , \ \Vert y_{i} \Vert =1, \ \text{for} \ i=1,2 \sem \text{and }\sem
y_i  \in B_{\ell_1} \sep \text{for } \sep 1 \le i \le 4.
\end{equation}
For each $1 \le i \le 4$ we  also have that
\begin{align}
\label{ci-ui2}
\nonumber
\| y_i - a_i \Vert & \leq \Big\| y_i - \frac{x_i}{1+8\rho} \Big\|  + \Big\|  \frac{x_i}{1+8\rho} - x_i \Big\|    + \| x_i-a_i \| \\
&< \left(1 - \frac{1}{1+8\rho}\right ) + \left(1 - \frac{1}{1+8\rho}\right ) + 4\rho \sem \text{(by \eqref{bi2})} \\
\nonumber
& = \frac{16\rho}{1+8\rho} + 4\rho\\
\nonumber
& < 20 \rho.
\end{align}
Now we  fix  $1\leq i_1 < i_2 < i_3 \leq 4$  and   estimate the norm of $\Vert  y_{i_1} - y_{i_2} +
y_{i_3} \Vert$ as follows.  In case that  $\{1,2\} \subset \{  i_1, i_2, i_3 \}$ we have
\begin{align}
\label{ci1-ci2+ci32}
\nonumber
\| y_{i_1} - y_{i_2} + y_{i_3} \Vert & = \Big\| \frac{x_{i_1}-x_{i_2}+x_{i_3}}{1+8\rho} \Big\|  \\
%   \nonumber
&\leq  \frac{\|a_{i_1}-a_{i_2}+a_{i_3}\| + \|x_1-a_1\| +  \|x_2-a_2 \|}{1+8\rho}\\
\nonumber
& <\frac{\|a_{i_1}-a_{i_2}+a_{i_3}\| + 8\rho}{1+8\rho} \seg \text{(by  \eqref{bi2})} \\
\nonumber
& \leq 1.
\end{align}
Otherwise  $\{1,2\}\not\subset  \{  i_1, i_2, i_3 \}$ and we obtain that
\begin{align}
\label{ci1-ci2+ci33}
\nonumber
\| y_{i_1} - y_{i_2} + y_{i_3} \Vert & \leq \Big\| \frac{x_{i_1}-x_{i_2}+x_{i_3}}{1+8\rho} \Big\|  + 1 -\frac{1}{1+8\rho}\\
&< \frac{\|a_{i_1}-a_{i_2}+a_{i_3}\| + 4\rho}{1+8\rho}+ \frac{8\rho}{1+8\rho} \sem \text{(by \eqref{bi2})} \\
\nonumber
& < 1 + 4 \rho.
\end{align}
Now we define
$$
z_i =
\begin{cases}
\frac{y_i}{1+ 4\rho} + (1-\frac{1}{1+4\rho})e_1         & \mbox{if} \sep  i \in \{ 1,2,3 \}   \\
\frac{y_4}{1+4\rho}   & \mbox{if} \sep   i =4.
\end{cases}
$$
In view of \eqref{ci-second-case} we have that
\begin{equation}
\label{yi-second-case}
z_{i}\geq 0, \sep u^*(z_i)=1  \sep \text{for } \sep  i=1,2  \sem \text {and}  \sem z_i \in B_{ \ell_1} \sep  \text{for} \sep 1 \le i \le 4.
\end{equation}
For each  $1\leq i \leq 4$  we obtain  that
\begin{align}
\label{yi-ui-second-case}
\nonumber
\| z_i - a_i \Vert & \leq \Big\| z_i - \frac{y_i}{1+4\rho} \Big\|  + \Big\|  \frac{y_i}{1+4\rho} - y_i \Big\|    + \| y_i-a_i \| \\
&<\left(1 - \frac{1}{1+4\rho}\right ) + \left(1 - \frac{1}{1+4\rho}\right ) + 20\rho \sem  \text{(by \eqref{ci-second-case}  and  \eqref{ci-ui2})} \\
\nonumber
& = \frac{8\rho}{1+4\rho} + 20\rho < 28 \rho < \varepsilon.
\end{align}
Now we check that  $(z_i)_{i \le 4} \in M_{\ell_{1}}^4$.
For each $1\le i_1 < i_2 < i_3 \leq 4$, we consider the following two cases. If  $\{ 1,2\} \subset \{  i_1, i_2, i_3 \}$ it is clear that
\begin{align}
\nonumber
\| z_{i_1} - z_{i_2} + z_{i_3} \Vert & \leq  \frac{\| y_{i_1}-y_{i_2}+y_{i_3}\| }{1+4\rho} +  1 - \frac{1}{1+4\rho} \\
\nonumber
& < \frac{1}{1+4\rho}+ 1 - \frac{1}{1+4\rho} \seg \text{(by \eqref{ci1-ci2+ci32})} \\
\nonumber
& = 1.
\end{align}
Otherwise    $\{ 1,2\} \not\subset \{  i_1, i_2, i_3 \}$ and we obtain  that
\begin{align}
\nonumber
\| z_{i_1} - z_{i_2} + z_{i_3} \Vert & =  \frac{\| y_{i_1}-y_{i_2}+y_{i_3}\| }{1+4\rho}  \\
\nonumber
& < 1 \seg \text{(by   \eqref{ci1-ci2+ci33})}.
\end{align}
In view of \eqref{yi-second-case} and  \eqref{yi-ui-second-case}, since  we checked that
$(z_i)_{  i \le 4} \in M_{\ell_{1}}^4$,  the proof is finished  in  case 1.

\noindent  $\bullet$ \underline{Case 2:} Assume  that  $C'=\{ 1,2,3\}$. We know that
\begin{align}
\nonumber
1-4\rho & = 2(1-2\rho) -1< u^{*}(b_1 + b_3) - u^{*}(b_2) \sem \text{(by  \eqref{ai-big})}  \\
\nonumber
&= u^{*}(b_1- b_2 + b_3) \le \Vert b_1-b_2 + b_3 \Vert \\
\nonumber
& \le \Vert a_1-a_2 + a_3\Vert +  \sum _{ i=1}^ 3 \Vert b_{i}-a_{i} \Vert\\
\nonumber
& < 1 + 6\rho  \seg \text{(by  \eqref{ai-ui})}.
\end{align}
We obtained that
\begin{equation}
\label{ai0j0k0-case3}
1-4\rho <  u^{*}(b_1-b_2 + b_3) \leq  \Vert b_1-b_2 + b_3\Vert < 1 + 6\rho.
\end{equation}
We can apply  Lemma \ref{lema-auxiliar} with $b_1$ playing the role of $z$,  so there is  $x_1\in \ell_{1}$ such that
\begin{align}
\label{w1}
x_1\ge b_1 \ge 0, \sem  \Vert x_1-b_1 \Vert \le 10\rho \sep \text{and} \sep x_1-b_2+b_3 \ge 0.
\end{align}
We define
$$
x_i=b_i\sep  \text{for}\sep  i\in \{ 2,3\} \sem \text{and}  \sem x_4=a_4.
$$
Notice that in view of  \eqref{ai-ui} and  \eqref{w1} we have
\begin{equation}
\label{wi-ui3}
\Vert x_i - a_i \Vert  <  12\rho, \sem  \forall i \le 4.
\end{equation}
As a consequence, for each   $i\in \{1,2,3\}$, from  \eqref{ai-big} and  \eqref{w1}    it follows that
\begin{align}
\label{3wi}
\nonumber
1-2\rho& < \Vert b_i\Vert  \\ &\le\Vert x_i \Vert \\ \nonumber& \le \Vert x_i - a_i \Vert + \Vert a_i \Vert < 1+ 12\rho.
\end{align}

If  $ 1\le i_1 < i_2 < i_3 \le 4$  we obtain that
\begin{align}
\label{wi1-wi2-3-case}
\Vert x_{i_1}-x_{i_2}+x_{i_3}\Vert & \le \Vert  a_{i_1}-a_{i_2}+a_{i_3}\Vert + \sum _{ j=1}^3 \Vert x_{i_j} - a_{i_j}\Vert \\
\nonumber &  <1+ 36\rho \seg  \text{(by  \eqref{wi-ui3})}.
\end{align}
Now we define a new element $(y_i)_{ i \le  4}$  in order to have that  $\{y_i: i\le 3 \} \subset \{x \in B_{\ell_{1}}: u^*(x)=1\}$.  We put 
$$
y_i =
\begin{cases}
\frac{x_i}{1+36\rho}   + ( 1-   \frac{\Vert x_i\Vert}{1+36\rho})e_1       & \mbox{if } \sep i\in \{1,2,3\}  \\
\frac{x_4}{1+36\rho}  & \mbox{if} \sep \, \, i=4.
\end{cases}
$$
For each $i \le 3$, since $b_i \ge 0$, from 
\eqref{w1}  we also have that $x_i \ge 0$. In view of \eqref{3wi} we deduce that 
%Since $b_i \ge 0$ for each $i \le 3$, from %\eqref{w1}  and \eqref{3wi} it is clear that
\begin{align}
\label{zi}
y_i \ge 0,  \sem u^{*}(y_i)=1 \sep \text{for }\sep i\in \{1,2,3\} \sep \text{and} \sep
\Vert y_4 \Vert \le 1.
\end{align}
For each $1 \le i \le 4$, in view of    \eqref{3wi} and \eqref{wi-ui3}    we obtain that
\begin{align}
\label{zi-ui}
\nonumber
\Vert y_i-a_i \Vert & \le \Big\Vert y_i -\frac{x_i}{1+36\rho}\Big\Vert + \Big\Vert \frac{x_i}{1+36\rho} - x_i \Big\Vert + \Vert x_i-a_i\Vert \\
& < \Bigl(1 -\frac{1-2\rho }{1+36\rho}\Bigr) + \Vert x_i\Vert\Bigl(1 -\frac{1}{1+36\rho} \Bigr) + 12\rho \\
\nonumber
& < \frac{38 \rho  }{1+36\rho} + (1+12\rho)\frac{36\rho}{1+36\rho} + 12\rho \\
\nonumber
& < 86\rho.
\end{align}
On one hand,  since $x_i \ge 0$ for each  $i \le 3$  we also have that
\begin{align}
\label{z1-z2}
\nonumber
\Vert y_1-y_2+y_3\Vert & = \Bigl\Vert  \frac{x_1-x_2+x_3}{1+36\rho} + \Bigl( 1 - \frac{\Vert x_1\Vert -
	\Vert x_2\Vert +\Vert x_3 \Vert  }{1+36\rho}\Bigr)e_1\Bigl\Vert \\
&= \Bigl\Vert  \frac{x_1-x_2+x_3}{1+36\rho} +
\Bigl( 1 - \frac{\Vert x_1 - x_2 + x_3\Vert  }{1+36\rho}\Bigr)e_1\Bigl\Vert  \sem \text{(by \eqref{w1})} \\
\nonumber
&= \Bigl\Vert  \frac{x_1-x_2+x_3}{1+36\rho}\Bigl\Vert +  \Bigl( 1 - \frac{\Vert x_1 - x_2 + x_3
	\Vert  }{1+36\rho}\Bigr) \sem \text{(by \eqref{wi1-wi2-3-case})} \\
\nonumber
&=1.
\end{align}
On  the other hand, if  $1\le i_1 < i_2 <  4$  then
\begin{align}
\label{zi1-zi2}
\nonumber \Vert y_{i_1}-y_{i_2}+y_4\Vert &\le \Bigl\Vert  \frac{x_{i_1}-x_{i_2}+x_4}{1+36\rho} \Bigl\Vert + \frac{1}{1+36\rho}\Bigl |\Vert x_{i_1} \Vert- \Vert x_{i_2} \Vert \Bigr  |\\
& < 1+    \frac{14\rho}{1+36\rho} \sem \text{(by \eqref{3wi}  and \eqref{wi1-wi2-3-case})} \\
\nonumber
& < 1 + 14\rho.
\end{align}

Now we  define the element in $M_{\ell_1}^4$ satisfying the required conditions in this case. To this purpose we take
$$
z_i =
\begin{cases}
\frac{y_i}{1+14\rho}   + ( 1-   \frac{1}{1+14\rho})e_1       & \mbox{if}\sep i\in \{1,2,3\}  \\
\frac{y_4}{1+14\rho}  & \mbox{if} \sep i=4.
\end{cases}
$$
Let us fix $1\le i\le 4$. In view of \eqref{zi}   it is satisfied that
\begin{equation}
\label{yi-case3}z_i \ge 0,  \sep  \sep u^{*}(z_i)=1 \sep \text{for} \sep i\in \{1,2,3\} \sep
\text{and} \sep \Vert z_4 \Vert \le 1.
\end{equation}
By using   \eqref{zi} and   \eqref{zi-ui} we obtain the following upper estimate
\begin{align}
\label{yi-ui-case3}
\nonumber
\Vert z_i - a_i  \Vert & \le  \Big \Vert z_i - \frac{y_i}{1+14\rho}\Big \Vert + \Big \Vert  \frac{y_i}{1+14\rho} - y_i  \Big \Vert + \Vert y_i -a_i \Vert \\
& < \Bigl( 1-   \frac{1}{1+14\rho}\Bigr) + \Vert y_i\Vert \Bigl( 1-   \frac{1}{1+14\rho}\Bigr) +86\rho \\
\nonumber
& < 28\rho +86\rho = 114\rho <\varepsilon.
\end{align}
We also have that
\begin{align}
\label{yi0j0k0-case3}
\nonumber
\Vert z_1-z_2+z_3 \Vert & = \Big \Vert \frac{y_1-y_2+y_3}{1+14\rho} +\Bigl( 1-   \frac{1}{1+14\rho}\Bigr)e_1 \Big \Vert \\
&\leq \frac{\Vert y_1-y_2+y_3 \Vert}{1+14\rho} + \Bigl( 1-   \frac{1}{1+14\rho}\Bigr)    \\
\nonumber
&= 1  \seg  \text{(by  \eqref{z1-z2})}.
\end{align}
In case that $1\le i_1< i_2< 4$  we obtain that
\begin{align}
\label{yi1-i2-i3-case3}
\Vert z_{i_1} - z_{i_2}+z_4\Vert & = \Big \Vert \frac{y_{i_1}-y_{i_2}+y_4}{1+14\rho} \Big \Vert  \\
\nonumber
&< 1  \seg  \text{(by  \eqref{zi1-zi2})}.
\end{align}

From equations \eqref{yi-case3}, \eqref{yi-ui-case3}, \eqref{yi0j0k0-case3} and \eqref{yi1-i2-i3-case3} the element  $( z_i)_{ i \le 4}\in M_{\ell_1}^4$ satisfies  all the required conditions. 

\noindent  $\bullet$ \underline{Case 3:}  Assume that  $C'=\{1,2,3,4\}$.
%If  $1\le i_1<i_2<i_3\le 4$ then
%\begin{align}
%\label{ai1-ai2}
%\nonumber
%1-4\rho & = 2(1-2\rho) -1< u^{*}(b_{i_1} + %b_{i_3}) - u^{*}(b_{i_2})  \seg  \text{(by  %\eqref{ai-big})}\\
%\nonumber
%&= u^{*}(b_{i_1} - b_{i_2} + b_{i_3}) \le %\Vert b_{i_1} - b_{i_2}+ b_{i_3}\Vert \\
%\nonumber
%& \le \Vert a_{i_1} - a_{i_2} + a_{i_3}\Vert %+  \sum _{j=1}^3 \Vert b_{i_j}-a_{i_j} \Vert %\\
%\nonumber
%& < 1 + 6\rho  \seg  \text{(by  %\eqref{ai-ui})}.
%\end{align}
For each $1 \le i_1 < i_2 < i_3 \le 4$, by the same argument  used to obtain \eqref{ai0j0k0-case3} in case 2, with $(i_1, i_2, i_3)$ playing the role of $(1,2,3)$ there we get that 
\begin{equation}
    \label{b123-case3}
1-4\rho < u^{*}(b_{ i_1} - b_ {i_2} + b_ {i_3}) \le \Vert b_{ i_1} - b_ {i_2} + b_ {i_3} \Vert < 1 + 6\rho .
\end{equation}
In view of  Lemma  \ref{lema-auxiliar},  there are $x_1, x_4 \in \ell_1$ such that
\begin{align}
\label{v1-a1}
x_i \ge b_i, \seg  &   \Vert x_i - b_i \Vert  \le 10\rho \sep \text{for} \sep i=1,4, \\
\nonumber
x_1 - b_2 + b_3 \ge0 \sem \text{and} &  \sem b_1 - b_3 + x_4\ge 0.
\end{align}
As a consequence
\begin{align*}
1-4\rho &< u^{*}(b_1 - b_2 + b_4) \le  u^{*}(x_1 - b_2 + b_4)  \seg  \text{(by  \eqref{b123-case3} and \eqref{v1-a1})}  \\
& \le \Vert x_1 - b_2 + b_4 \Vert \\
&\le \Vert b_1 - b_2 + b_4 \Vert  + \Vert x_1 - b_1 \Vert \\
& < \Vert a_1 - a_2 + a_4 \Vert  + 6 \rho  + \Vert x_1 - b_1 \Vert    \seg  \text{(by  \eqref{ai-ui})}      \\
& \le  1 + 16\rho    \sem  \text{(by   \eqref{v1-a1})}
\end{align*}
and
\begin{align*}
1-4\rho &< u^{*}(b_2 - b_3 + b_4) \le  u^{*}(b_2 - b_3 + x_4)  \sem  \seg  \text{(by  \eqref{b123-case3} and \eqref{v1-a1})} \\
& \le \Vert b_2 - b_3 + x_4 \Vert \\
&\le \Vert b_2 - b_3  + b_4 \Vert  + \Vert x_4 - b_4 \Vert  < 1 + 16 \rho
  \seg  \text{(by  \eqref{ai-ui}   and  \eqref{v1-a1})}.
\end{align*}
From the last two chains of inequalities we have that
\begin{align}
\nonumber  1-4\rho < u^{*}(x_1 - b_2 + b_4) \le \Vert x_1 - b_2 + b_4 \Vert
< 1 + 16\rho   \\
\nonumber  1-4\rho < u^{*}(b_2 - b_3 + x_4) \le \Vert b_2 - b_3 + x_4 \Vert < 1 + 16\rho.
\end{align}
By applying  again  Lemma \ref{lema-auxiliar},  there are   $y_1, y_4 \in \ell_1$  satisfying that
\begin{align}
\label{wi-vi}
y_i \ge x_i, \seg  &   \Vert y_i - x_i \Vert  \le 20\rho \sep \text{for} \sep i=1,4, \\
\nonumber
y_1 - b_2 + b_4 \ge0 \sem \text{and} &  \sem b_2 - b_3 + y_4\ge 0.
\end{align}
Now we define $y_2=b_2$ and  $y_3=b_3$.

By using   \eqref{wi-vi}, \eqref{v1-a1}  and \eqref{ai-ui}, for  each  $ i\le 4$ it is clear that  
\begin{align}
\label{wi-ui}
\Vert y_i - a_i \Vert &\le \Vert y_i -  b_i  \Vert + \Vert b_i - a_i\Vert \\
\nonumber
& < 20\rho + 10\rho + 2\rho =32\rho.  
\end{align}
As a consequence, from    \eqref{v1-a1} and  \eqref{wi-vi} for each $i \le 4$  we also have that 
\begin{align}
\label{wi}
1-2\rho &< u^{*}(a_i)\le u^{*}(b_i) \le u^{*}(y_i) \le \Vert y_i\Vert \le \Vert y_i - a_i \Vert  + \Vert  a_i \Vert < 1 + 32\rho.
\end{align}

In view of \eqref{v1-a1} and \eqref{wi-vi} it is immediate to check that
\begin{equation}
\label{wi-wj}
y_{i_1} -y_{i_3} +y_{i_3} \ge 0 \sem \text{for any } \sep 1 \le i_1 < i_2 < i_3 \le 4.
\end{equation}

For every $ 1\le i_1<i_2<i_3\le 4$  we also get that
\begin{align}
\label{wi-wj-wk-case-3}
\Vert y_{i_1} - y_{i_2} + y_{i_3} \Vert &\le \Vert a_{i_1} - a_{i_2} + a_{i_3} \Vert + \sum _{j=1}^3 \Vert y_{i_j} - a_{i_j} \Vert \\
\nonumber
& < 1 + 96 \rho  \seg  \text{(by  \eqref{wi-ui})}.
\end{align}

Now we define the elements satisfying the required conditions  by
$$
z_i = \frac{y_i}{1+96\rho} + \Bigl(  1-\frac{\Vert y_i\Vert}{1+96\rho}  \Bigr)e_1   \seg   (1 \le i \le 4).
$$
By \eqref{wi-vi} and \eqref{v1-a1} it is clear that $y_i \ge 0$ for each $ i \le 4$.   From   \eqref{wi} we know that $\Vert y_i \Vert \le 1 + 32 \rho$ and so 
we deduce that 
\begin{equation}
\label{yi-case4}
z_i \ge 0   \seg  \text{and}  \seg  u^{*}(z_i)=1, \seg \forall 1\le i \le 4.
\end{equation}
We also obtain that
\begin{align}
\label{yi-ui-case4}
\nonumber
\Vert z_i - a_i \Vert & \le  \Big \Vert z_i - \frac{y_i}{1+96\rho} \Big \Vert +
\Big \Vert \frac{y_i}{1+96\rho} - y_i  \Big \Vert + \Vert y_i -a_i \Vert \\
& < \Bigl(  1-\frac{\Vert y_i\Vert}{1+96\rho}  \Bigr) +
\Vert y_i\Vert \Bigl(  1-\frac{1}{1+96\rho}  \Bigr) + 32\rho
\seg  \text{(by  \eqref{wi-ui})} \\
\nonumber
& <  \Bigl(  1-\frac{1-2\rho}{1+96\rho}  \Bigr) + \frac{1+32\rho}{1+96\rho}96\rho
+32\rho \seg \seg  \text{(by  \eqref{wi})} \\
\nonumber
& <   \frac{98\rho}{1+96\rho} + 96\rho + 32\rho \\
\nonumber
& < 226\rho=\varepsilon.
\end{align}
Finally, notice that for every  $1\le i_1<i_2<i_3\le 4$, since $y_i \ge  0$ for every $i \le 4 $  we have that
\begin{align}
\nonumber
\Vert z_{i_1} - z_{i_2} + z_{i_3}\Vert & = \Big \Vert \frac{y_{i_1} - y_{i_2} + y_{i_3}}{1+96\rho} +
\Bigl(  1-\frac{\Vert y_{i_1}\Vert - \Vert y_{i_2}\Vert +\Vert y_{i_3}\Vert}{1+96\rho}  \Bigr) e_1\Big \Vert \\
\label{y1-y2-y3-case4}
&=\Big \Vert \frac{y_{i_1} - y_{i_2} + y_{i_3}}{1+96\rho} +
\Bigl(  1-\frac{\Vert y_{i_1} -  y_{i_2}+ y_{i_3}\Vert}{1+96\rho}  \Bigr) e_1\Big \Vert
\seg  \text{(by   \eqref{wi-wj})} \\
\nonumber
&\le  \frac{\Vert y_{i_1} - y_{i_2} + y_{i_3}\Vert }{1+96\rho} +
1-\frac{\Vert y_{i_1} -  y_{i_2}+ y_{i_3}\Vert}{1+96\rho}
\seg  \seg  \text{(by  \eqref{wi-wj-wk-case-3})} \\
\nonumber
&=1.
\end{align}
In view of equations \eqref{yi-case4}, \eqref{yi-ui-case4} and \eqref{y1-y2-y3-case4}
the proof is also finished  in case 3. So we proved the statement for $\ell_1$.

The proof for $\ell_{1} ^n$  follows from the same  argument  that we used for $\ell_1$ by considering
elements in  $M_{\ell_{1}^n}^4$ (instead of elements in    $M_{\ell_{1}}^4$)  and  the description $\ext (B_{(\ell_{1}^n)^*  }) = \{ z^* \in (\ell_{1}^n)^* : \vert  z^* (e_k) \vert =1, \forall k \le n\}$.
\end{proof}

\vskip10mm

%{\bf Because of Theorem 2.1 in  Aron et all (TAMS, label ACKLM),  and our %Characterization (Theorem 3.3) the statement for  $\ell_1 ^n$ follows from the %statement for $\ell_1$. However, I do not think that it deserves to be noticed. What %do you think?
	
	%In case that you do think that this needs a reference, we can change a little the %last paragraph in the proof.

%I BELIEVE that there is no need to say that

%}

%\vskip10mm

Next result  follows from  the same argument of \cite[Theorem 2.7]{ABGKM2}. It will be useful, for instance, to extend the previous result 
to $L_1  (\mu)$ for any positive measure $\mu$.

\begin{theorem}
Assume that  $Y$ is a Banach space and $\gamma : ]0, 1[ \llll \R^+$ is a function such that
\linebreak[4]
  $Y=\overline{\cup \{Y_\alpha : \alpha \in \Lambda  \}}$, where $\{Y_\alpha : \alpha \in \Lambda  \}$  is a nested family of subspaces of $Y$ satisfying the AHSp-$\ell_\infty^4$ with the function $\gamma$. Then $Y$ has  the AHSp-$\ell_\infty^4$ with the function  $\zeta (\varepsilon) =  \gamma \bigl( \frac{\varepsilon}{2} \bigr)$.
\end{theorem}

\begin{proof} 
Given $0<\varepsilon<1$,  let $\gamma(\varepsilon )$ be the positive real number satisfying Definition \ref{def-AHSPL4} for each space $Y_\alpha $.

Assume that  $(a_i)_{ i \le 4} \in M_{Y} ^4 $   and  that for some   nonempty set  $A \subset  \{ 1,2,3,4 \}$ and  $y^{*} \in S_{Y^{*}}$, it is satisfied that  $y^{*}(a_i)>1-\gamma \bigl(\frac{\varepsilon}{2} \bigr) $ for each  $i\in A$. Let us choose a  real  number $t$ such that
$$
0<t< \frac{1}{4}\min  \Bigl \{ \frac{ \varepsilon}{2}, \min \Bigl \{ y^*(a_i)-1+ \gamma \Bigl(\frac{\varepsilon}{2} \Bigr): i\in A \Bigr\} \Bigr\}.
$$ 
By assumption there exist $\alpha_0\in \Lambda$  and $\{b_i: i \leq 4\} \subset
B_{Y_{\alpha_0}}$ satisfying
$$
\Vert  b_i - a_i \Vert  < t\seg \ \ \text{for all } i\le 4.
$$
Now we define $y_i=\frac{b_i}{1+3t}$ for any $i\leq 4$.  By using that $(a_i) _{ i \le 4} \in M_Y ^4 $ it  is  immediate to check that  $(y_i)_{i \le 4} \in M_{Y_{\alpha _0}^4}$. We clearly have 
\begin{equation}
\label{y*yi}
\Vert  y_i-a_i  \Vert  \le \Bigl \Vert  \frac{b_i}{1+3t} - b_i \Bigr \Vert +  \Vert b_i-a_i \Vert 
< 3t+t= 4 t<\frac{\varepsilon}{2}, \sep  \text{for all} \sep i \le 4.
\end{equation}
For each $i \in A$ we obtain that  
\begin{equation}
\label{Ny}
y^*(y_i) >y^*(a_i)-4 t>1-\gamma \Bigl( \frac{\varepsilon}{2} \Bigr)>0.
\end{equation}

We define the element $z^* \in Y_{\alpha_0}^{*}$  by 
$$
z^*(z)=y^*(z) \seg \ \  (z \in Y_{\alpha_0}).
$$
Since $\Vert z^* \Vert \le  \Vert y^* \Vert $, we get that  $z^* \in B_{ Y_{\alpha _0}^* }$. In view of \eqref{Ny}  we know that $z^* \ne 0$ and we also have that
$$
\frac{z^*}{\|z^*\|}(y_i)=\frac{y^*}{\|z^*\|}(y_i)\geq y^*(y_i)>1-\gamma \Bigl( \frac{\varepsilon}{2} \Bigr)\seg \ \ \text{for all } i\in A.
$$
By assumption there  is  $(z_i)_{i\le 4} \in  M_{Y_{\alpha_0}}^4$  such that
    \begin{enumerate}
                \item[i)]  $\|z_i-y_i\|<\frac{\varepsilon}{2}  \sep \text{for each  }\sep  i \le 4,$
        \item [ii)]  $ \bigl \Vert \sum _{ i \in A} z_i \bigr\Vert  = \vert A \vert $.
    \end{enumerate}
In view of  \eqref{y*yi} we obtain that 
$$
\Vert z_i - a_i \Vert \le   \Vert  z_i -  y_i \Vert +
\Vert  y_i - a_i \Vert < \frac{\varepsilon}{2} + \frac{\varepsilon}{2} =
\varepsilon, \seg \forall i \le 4.
$$
We proved that $Y$ satisfies Definition \ref{def-AHSPL4} with the function $\zeta (\varepsilon) = \gamma \bigl( \frac{\varepsilon}{2} \bigr)$.
\end{proof}

%\newpage

Let us remark that the subspace of  $L_1 (\mu)$ generated by a finite number of characteristic functions is linearly isometric to some space $\ell_1^n$. Since the space of simple functions is dense in $L_1 (\mu)$, in view of Proposition \ref{prop-ell1}, we conclude that  $L_1 (\mu)$ satisfies the assumption of the previous result  and the function $\gamma $  does not depend on $\mu$.  In view of Theorem \ref{teo-char}  we obtain that the pairs $(\ell_\infty ^4, L_1 (\mu))$ have uniformly BPBp for operators for every measure $\mu$. More concretely we deduce the following assertions.

\begin{corollary}
There is a function  $\gamma: ]0,1[ \llll  \R^+$ such that  $L_1 (\mu)$ satisfies Definition \ref{def-AHSPL4} with such function, for any positive measure $\mu$. Hence the pair $(\ell_\infty^4, L_1 (\mu) )$ has the BPBp for operators. Moreover there is a function $\eta $ such that  the pair   $(\ell_\infty^4, L_1 (\mu) )$ satisfies Definition \ref{def-BPBp} for such function,   for any positive measure $\mu$.
\end{corollary}

\end{document}